\documentclass[english]{article}
\usepackage{amsmath,amssymb,latexsym,amsthm,mathrsfs,appendix}
\usepackage{color}

\usepackage{url}

\date{}

\def\theenumi{\arabic{enumi}}

\def\theenumii{\alph{enumii}}
\def\p@enumii{\theenumi.}

\def\theenumiii{\arabic{enumiii}}
\def\p@enumiii{(\theenumi)(\theenumii)}

\def\p@enumiv{\p@enumiii.\theenumiii}
\newtheorem{theorem}{Theorem}[section]

\newtheorem{corollary}[theorem]{Corollary}

\newtheorem{proposition}[theorem]{Proposition}
\newtheorem{lemma}[theorem]{Lemma}

\theoremstyle{definition}
\newtheorem{example}[theorem]{Example}

\newtheorem{notation}[theorem]{Notation}

\newtheorem{definition}[theorem]{Definition}
\newtheorem{remark}[theorem]{Remark}

\newcommand{\ra}{\rightarrow}

\newcommand{\id}{\mbox{id}} \newcommand{\idm}{\mbox{\em id}}  
 \newcommand{\ids}{\mbox{\small id}}  \newcommand{\idt}{\mbox{\tiny id}}  
 
  \makeatother

\usepackage{babel}
\begin{document}

\title{Isoperimetric profiles and random walks on some groups defined by piecewise actions
}
\author{Laurent Saloff-Coste\thanks{%
Partially supported by NSF grant DMS 1404435 and DMS 1707589} \\
{\small Department of Mathematics}\\
{\small Cornell University} \and Tianyi Zheng \\
{\small Department of Mathematics}\\
{\small University of California San Diego} }
\maketitle

\begin{abstract}
We study the isoperimetric and spectral profiles of certain families of finitely 
generated  groups defined via actions on labelled Schreier graphs and simple {\em gluing}  of such. In one of our simplest constructions---the {\em pocket-extension} of a group $G$---this leads to the study of certain finitely generated subgroups of the full permutation group  $\mathbb S(G\cup \{*\})$. Some sharp estimates are obtained while many challenging questions remain.
 \end{abstract}

This work is dedicated to the memory of Harry Kesten who, among his many outstanding contributions to mathematics, initiated the study of random walks on groups.

\section{Introduction}\setcounter{equation}{0}

\subsection{A short historical perspective} The term {\em random walk} was introduced in a short note in the form of a question  that Karl Pearson sent to the journal Nature in 1905. 
The random walk in question took place in the plane.  In the second edition of {\em Le Calcul des Probabilit\'es}, published in 1912, Henri Poincar\'e  discusses the mixing of cards produced by repeated shuffles and explains how it is modeled by repeated multiplications of random elements in a permutation group. In 1921, 
Geoge P\'olya famously considered the dichotomy between recurrence and transience in the context of simple random walk on a $d$-dimensional grid. In 1940,
Kioshi Ito and  Yukiyosi Kawada considered repeated convolutions on compact groups. By the 1950's, the concept of random walks in $d$-dimensional space and on discrete  lattices in $d$-space was well established.  In 1950, Mark Kac won his first of two Chauvenet Prizes for an article titled {\em Random Walk and the Theory of Brownian Motion} published in the {\em Monthly} three years earlier. Also in 1950, Dvoretzky and Erd\"os wrote {\em Some Problems on Random Walk in Space} for the second Berkeley Symposium on Mathematics Statistics and Probability.  Except for sporadic interest in card shuffling problems and a few other isolated works, it is hard to find any reference before 1958 were multiplying random elements of a non-commutative group is mentioned. 

In the summer of 1956, Harry Kesten--who was then a student in Amsterdam--wrote to Mark Kac. He asked if he could come to Cornell to work under Kac's supervision. A fellowship was offered and Kesten came to Cornell that fall. In the spring of 1958, he defended his thesis titled {\em Symmetric Random Walks on Groups.}    The first sentence reads:\\ 

 {\em Let G be a countable group and let $A = \{a_1,a_2,\dots\}$ ($a_i\in G$) generate G. Consider the random walk on G in which every step consists
of right multiplication by $a_i$ or its inverse $a_i^{-1}$, each with probability $p_i $ ($p_i \ge0$, $2 \sum_i p_i= 1$).}\\

Kesten goes on to explain that the paper is concerned with the relations between properties of the highest eigenvalue $\lambda(G)$ of the associated linear operator on $\ell^2(G)$ and the structure of the group (Kesten uses the shorthand notation $\lambda(G)$ when it is clear which random walk is considered). The final section (section 5) discusses some open problems including the following:\\

{\em As mentioned in \S 3, it would be interesting to
find all groups with $\lambda(G) = 1$. Especially, since for every finite group, the
spectrum contains 1. A weak form of the Burnside conjecture would be:
 ``If G is finitely generated and every element has bounded (or more general,
finite) order, then $\lambda(G) = 1$.'' This would readily follow if one could prove the
converse of Corollary 3, i.e., ``If G has no free subgroups on more than 1
generator, then $\lambda(G) = 1$.'' However, the author was unable to prove or disprove
this. If this converse of Corollary 3 is not true, however, it might be
possible to construct a group G in which every element has finite order but
$\lambda(G) <1$.}\\

The paper end with:\\

{\em Note added in proof. Since this paper was submitted, the author proved
that $\lambda(G) = 1$ is equivalent to the existence of an invariant mean on G (cf.
Full Banach mean values on countable groups, Math. Scand. vol. 7 (1959)).
It seems that the Burnside conjecture has been disproved recently in
Russia. }\\

These direct quotes form Kesten's paper leave no doubt that his work introduced the subject of random walks on groups with a  strong emphasize on (a) understanding random walks on groups in general and (b) understanding the relations between the behavior of random walks and the structure of the underlying group. The last sentence of the ``note added in proof'' refers to P.S. Novikov's 1959 announcement that the Burnside group $B(m, n)$ is infinite for $n$ odd, $n > 71$. This announcement was premature. Only in 1968 did Adyan and Novikov succeed to prove that $B(m, n)$ is infinite for $n$ odd, $n > 4381$. See \url{http://www-gap.dcs.st-and.ac.uk/~history/HistTopics/Burnside_problem.html}
for a brief history of the Burnside Problem.\\

The first author learned directly from his colleague Harry Kesten that  the subject of Kesten's Ph.D. thesis evolved from a very specific question suggested by
Mark Kac. This question was perhaps related to a problem considered in 1954 by Richard Bellman in {\em Limit theorems for non-commutative operations. I.} \cite{Bellman}. The Mathematical Reviews entry for this article was written by J. Wolfowitz, one of the Cornell faculty who interacted with Kesten during his time as a graduate student. Kac's question concerned the behavior of some sort of dynamics that switches randomly between two 2 by 2 matrices. There is no trace of this question in Kesten's thesis although he would come back to the related problem of the study of products of random matrices in his famous joint work with H. Furstenberg which was initiated when both where visiting Princeton in 1958/59.  

It is notable that Kesten's thesis does not introduce random walks on groups as a generalization of random walks on the $d$-dimensional 
grid. The text gives no references to such works (e.g., no references to P\'olya and subsequent works). It is also curious that the paper does not use the term convolution at all and only rarely appeals explicitely to the multiplication law of the group!  E.B. Dynkin and M.B. Malyutov (1961) and G. Margulis (1966) wrote important related papers in the following years. Neither cites Kesten's thesis but Kesten reviewed both papers for Mathematical Reviews.  The famous little book {\em Probabilities on Algebraic Structures} published by Ulf Grenender  in 1963 gives only marginal attention to Kesten's work (Section 5.5.3 and related Note). 
 
It seems fair to say that Kesten's thesis did not immediately find an audience, especially in the probability community. During the 1960s, it attracted the attention of people interested in ergodic theory and non-commutative harmonic analysis and functional analysis (M.M. Day, H. Furstenberg), and of Marcel-Paul Sch\"utzenberger  who was interested in formal languages. Kesten himself became interested in random walks on abelian groups, a subject on which he collaborated with his Cornell colleague and close friend F. Spitzer. Kesten's contribution to the fifth Berkeley Symposium on Mathematics Statistics and Probability 
(1965), {\em The Martin Boundary of Recurrent Random Walks on Countable Groups},  is the rare instance when Kesten  revisited the subject he created in his thesis. The famous question known as 
{\em Kesten's Problem}---Which are the finitely generated groups that carry a non-degenerated
recurrent random walk?---Are those groups only the finite extensions of $\{0\},\mathbb Z$ and $\mathbb Z^2$?---emerged from this article.   

During the next decade (1970s), a group centered in France (A. Avez, E. Derriennic, Y. Guivarc'h, M. Keane and B. Roynette, encouraged by A. Brunel, and later, P. Baldi, Ph. Bougerol, and others) explored a variety of important questions around random walks on groups.  The volumes \cite{GKR,Nancy} give a representative picture of these efforts.  
In particular, Kesten's problem was resolved affirmatively in the context of connected Lie groups. The extend of the differences between this particular context and the context of finitely generated groups was perhaps not entirely apparent at the time. The contributions of J. Rosenblatt during the seventies should also be mentioned here.

It is during the 1980s that the subject of Random Walks on Groups took off thanks to remarkable progress and contributions. Using an amenability criterion based on co-growth developed by R. Grigorchuck (a criterion that parallels Kesten's amenability criterion), Adyan proved in 1982 that many Burnside groups are not only infinite but non-amenable. In 1983, V. Kaimanovich and A. Vershik  published an elegant and influential article in the Annals of Probability which, as the following quote makes clear, expands on Kesten original vision of the subject:\\

{\em Probabilistic properties of random walks on groups are deeply intertwined
with many essential algebraic characteristics of groups and their group algebras
(amenability, exponential growth, etc.). On the other hand, random walks on
groups regarded as a special class of Markov processes provide new simply describable 
examples of nontrivial probabilistic behavior. Both these aspects make
the subject especially interesting and important. }\\ 

M. Gromov had proved in 1981 that any finitely generated group with polynomial volume growth contains  a nilpotent group of finite index. A few years later, R. Grigorchuck proved that groups  of intermediate volume growth, that is, volume growth that is faster than any polynomial but slower than any exponential, exit and are, in fact, plentiful.  When, around 1985, N. Varopoulos established a sharp linked between volume growth of the type $V(r)\ge cr^d$ and the decay of the return probability of the type $\mathbf P_{\ids}(X_n=\id)\le Cn^{-d/2}$ (where $\id=\id_G$ denotes the identity element in $G$)
, he provided the solution to Kesten's problem: because of the recurrence criterion
$\sum \mathbf P_{\ids}(X_n=\id)=\infty$ and Gromov's theorem, the only finitely generated groups that carry a non-degenerate recurrent random walk are the finite extensions of $\{0\},\mathbb Z$ and $\mathbb Z^2$.  

Before these developments, the subject of Random Walk had been strongly influenced by areas of mathematics such as ergodic theory, harmonic analysis, representation theory, and the theory of Markov processes. This had left only a marginal role to what should have always been one of the main actors, group theory. Indeed, Lie groups and matrix groups--objects that are completely absent in Kesten's original work--had taken a preeminent role.  Little attention was given to finitely generated groups beyond the key example of free groups and a few other special cases. This changed drastically during the 1980s thanks in part  to the attention given to {\em geometric group theory} through the influential work of M. Gromov. For random walk theory, this had the momentous effect to bring back group theory --be it geometric group theory or combinatorial group theory--to the center of the stage.   

Here are some of the key interrelated questions that have emerged from this body of work:
\begin{itemize}
\item  What is the structure of sets of harmonic functions (bounded, positive, of polynomial growth, of a given growth type, slow of fast)?  Here, harmonic functions are solutions $u$  of the equation $u*\mu=u$ where $\mu$ is a given probability measure on $G$.
\item What are the spectral properties of the convolution operator  $f\mapsto f*\mu$ when $\mu$ is a (symmetric) probability measure? 
\item What is the behavior of the probability of return of a symmetric random walk driven by a measure $\mu$, $\mathbf P_{\ids}(X_n=\id)=\mu^{(n)}(\id)$, and, more generally, the behavior of $\max\{\mu^{(n)}(g): g\in G\}$ for non-symmetric measures?
\item What is the {\em escape behavior} of transient random walks captured, say,  in terms of some given distance function and in the form of  average displacement or almost sure results?  
\item What is the asymptotic entropic behavior, that is, the behavior of $n\mapsto \mathbf E(\log \mu^{(n)}(X_n))$ as $n$ tends to infinity?
\end{itemize}

In general, these questions can be phrased by asking: What is the influence of the structure of the group $G$ on the random walk behavior?  How does the answer depend on basic  properties of $\mu$  such as symmetry or moment assumptions?  Can some {\em random walk behaviors} (for classes of random walks on a group $G$) be deemed  {\em group invariant}? What properties of $G$ can be understood by observing random walk behaviors?  Can random walk behavior be used to understand groups better?  In each of these directions of research, many interesting natural questions remain open. 

We end this short historical perspective with some pointers to recent progress in the directions outlined above. 
Further references are found in the listed articles. 
Some recent results on harmonic functions and group structure are in \cite{AAMV,EPFWF,BEPF,KotV,AmirKozma,FHTV,amir2017exponential,ErZ}. 
Entropy and/or displacement are discussed in \cite{Erschlerdrift,AmBal,Amjoint,Naor2008,Lee2013,brieussel2015speed}. Probability of return, spectral and other properties are discussed in \cite{Bartholdi2005,BrieusselMZ,Amir2009,Bartholdi2010,Saloff-Coste2013b,SCZ-AOP2016,SCZ-Rev,ErschAlmost}.

The present work is devoted to the study of the behavior of random walks on groups that arise from a certain type of rather simple and basic combinatorial/algebraic construction. These groups are, in a natural particular way, subgroups of permutation groups on infinite countable vertex sets. Our key example is the  {\em pocket group} $G_{\circledast} $ built on the finitely generated group $G$. It
is the subgroup of $\mathbb S(G\cup \{*\})$ generated by all translations by elements in $G$ (by definition, these permutations fix $*$) and by the transposition   $\tau=(*,\id_G)$.  
\section{Spectral and isoperimetric profiles of pocket extensions}
 To any finitely generated group, one can associate
the monotone non-increasing functions 
$$\Lambda_{1,G},\Lambda_{2,G} \mbox{ and } \Phi_G$$ which, respectively, 
describe the $L^1$- and $L^2$-isoperimetric profiles 
and the return probability (or heat kernel decay) associated with 
the group $G$ (precise definitions are recalled below in Section \ref{sec-iso}). 
From a coarse analysis point of view which we briefly recalled below, 
these are group invariants in the 
sense that they do not depend on the particular choice of  the symmetric finite 
generating set that is used to define them. Celebrated Theorems due to F\o lner and Kesten assert that the 
dichotomy between amenable and non-amenable groups can be captured precisely 
using any one of these three invariants: A group is non-amenable if and only if 
$\Lambda_{1,G}$ (equivalently, $\Lambda_{2,G}$) is bounded below away from $0$,
and this is also equivalent to having $\Phi_G$ decay exponentially fast. 

This paper focuses on these 
invariants and how they depend on the structure of the underlying 
group in the context of several constructions which yield amenable groups based on the gluing of some basic actions.  
See Section \ref{sec-iso} for details.
To put this work in perspective, recall that among polycyclic groups or 
(almost equivalently) finitely generated 
discrete amenable subgroups of linear groups, the behaviors of 
$\Lambda_{1,G}$, $\Lambda_{2,G}$ and  $\Phi_G$ are well understood and fall 
in exactly 2 possible categories (the meaning of the notation $\simeq$ used below is spelled out at the beginning of Section \ref{sec-iso}): 
\begin{itemize}
\item The polycyclic group $G$ has exponential volume growth and
$$\Lambda_{1,G}(v)^2 \simeq \Lambda_{2,G}(v)\simeq  \frac{1}{[\log (1+v)]^2}
\;\mbox{ and } \Phi_G(n)\simeq \exp(- n^{1/3}).$$ 
\item The volume growth $V_G$ satisfies $V_G(r)\simeq (1+r)^d$ for 
some integer $d$ and
$$\Lambda_{1,G}(v)^2 \simeq \Lambda_{2,G}(v)\simeq (1+v)^{-2/d} 
\;\mbox{ and } \Phi_G(n)\simeq (1+n)^{-d/2}.$$ 
\end{itemize} 
These can be considered as the ``classical'' behaviors. 
See \cite{Tessera} for the description of a larger class of groups 
for which only these behaviors can occur. 

By now it is well-understood that, for more general groups, 
other behaviors can occur. See, e.g.,  
\cite{Erschler2006,Pittet2002,SCnotices,Varsol}.  One of the first and most popular example of construction that demonstrates the existence of other possible behaviors is the lamplighter group  $(\mathbb Z/2\mathbb Z)\wr G$ with base $G$.  Here $G$ is a finitely generated group and 
$(\mathbb Z/2\mathbb Z)\wr G $ is the semi-direct product
$$(\mathbb Z/2\mathbb Z)^{(G)} \rtimes_\alpha G$$
where $(\mathbb Z/2\mathbb Z)^{(G)}=\oplus_{g\in G} (\mathbb Z/2\mathbb Z)_g$ is the direct sum of countably many copies of 
$\{0,1\}=\mathbb Z/2\mathbb Z$ (i.e., $(\mathbb Z/2\mathbb Z)^{(G)}$ is the group of all binary sequences indexed by $G$ with finitely many non-zero entries). The action of $G$ on these binary sequences is by index translation (i.e., for $h\in G$,
$\alpha(h)((\eta_g)_{g\in G})=(\eta'_g)_{g\in G}$ with $\eta'_g=\eta_{h^{-1}g}$, $g\in G$).  This is a special case of a more general construction known as wreath product. To simplify notation, let $G^\wr=(\mathbb Z/2\mathbb Z)\wr G $ be the lamplighter group with base $G$. 

Works by A. Erschler \cite{Erschler2006}, by C. Pittet and the first author \cite{Pittet2002}, and by the present authors \cite{SCZ-AOP2016},  
describe how to compute the invariants $\Lambda_{1,G^\wr}$ (\cite{Erschler2006}), $\Lambda_{2,G^{\wr}}$, (\cite{SCZ-AOP2016}), and $\Phi_{G^\wr}$
(\cite{Pittet2002,SCZ-AOP2016}) as functions of the corresponding invariant for $G$.  In particular, 
$$\Lambda_{p,G^\wr} (v) \simeq \Lambda_{p,G} (\log (1+ v)),\;\;v>0,\;\;p=1,2 .$$

The goal of this paper is to provide similar results for a variety of related but different constructions.  Any countable group $G$ can be viewed as a subgroup of the group $\mathbb S(G)$ of all permutations of the set $G$. Namely, an element $h\in G$ is viewed as the permutation $g\mapsto hg$. The group 
$\mathbb S(G)$ is very large (not finitely generated and, indeed, uncountable) and it contains many finitely generated groups that contain $G$.
We are interested in certain of these finitely generated subgroups of $\mathbb S(G)$ which have  $G$ both as a subgroup and as a quotient, and which arise from some particular 
constructions that provide explicit generators.  One variant of this type of constructions---which we call the {\em pocket extension}---is as follows.  Let $G$ be a finitely generated group with identity element $\id_G$. Add a new element, $*$, to the countable set $G$ to form the set   $G\cup \{*\}$ (this is not a group). The pocket group $G_{\circledast} $
is the subgroup of $\mathbb S(G\cup \{*\})$ generated by all translations by elements in $G$ (by definition, these permutations fix $*$) and by the transposition   $\tau=(*,\id_G)$.  To understand what this means, view any element in $\mathbb S(G\cup\{*\})$ as a marking of $G\cup \{*\}$ by itself.  A translation by an element $h$ of $G$ leaves the marker at $*$ unchanged and move the marker  at  $g\in G$ to $hg\in G$. The transposition $\tau$ simply transposes the markers at $*$ and $e$.  This is consistent with the view that  permutations of a deck of cards are described by their action on the positions of the cards. 
We prove that
$$\Lambda_{p,G_{\circledast}} (v)\simeq \Lambda_{p,G}\left(\frac{\log (1+v)}{\log(1+\log(1+v))}\right), \;\; v>0,\;\;p=1,2$$
This is the same behavior as the known behavior for the wreath product $G\wr G$. 
Other constructions of this type are described in Section \ref{sec-gluing}.

\section{Preliminaries} \setcounter{equation}{0} \label{sec-Prel}

\subsection{Groups defined by labelled graphs and gluings} \label{sec-gluing}

In what follow,  $\mathbb S(X)$ denotes the full symmetric group of the set $X$ whereas $\mathbb S_0(X), \mathbb A_0(X)$ denote, respectively, the group of all permutations of $X$ with finite support and the alternating subgroup of permutations of $X$ with finite support which are of even type (signature $1$). 

Let $(X,E)$ be a  graph where the edge-set $E$ is equipped with  a map $ \phi: E\rightarrow X\times X$, $\phi(e)=(x(e),(y(e))$ (this map describes the edge $e$ as a pair of vertices, allowing for multiple edges and self loops). We assume that there is an involution $e\mapsto \check{e}$ with no fix points and such that
$\phi(\check{e})=(y(e),x(e))$.  Assume that $X$ is finite or countable and that $(X,E)$ is regular of degree $2k$ in the sense that for each $x\in X$ there are $2k$ edges in $E$ such that $x(e)=x$.
Let $\mathbf A=\{\boldsymbol \alpha_i: 1\le i\le k\}$ be an alphabet with abstract inverse $\mathbf A^{-1}=\{\boldsymbol{\alpha}_i^{-1}: 1\le i\le k\}$. 
A labelling of $(X,E)$ is a map $m: E\rightarrow \mathbf A\cup \mathbf A^{-1}$ such that $m(\check{e})= m(e)^{-1}$.  Call $(X,E, m)$ a labelled graph. 
\begin{figure}[h]
\begin{center}\caption{First step (Houghton group): $t_1$ in blue, $t_2$ in red. The color of a dot indicates the presence of a loop labelled with the associated generator.}\label{D1} 
\begin{picture}(300,220)(0,-20) 

{\color{blue}{
\multiput(149.5,77)(-15,0){10}{\vector(1,0){6}}}}
{\color{red}{
\multiput(146,73)(-15,0){10}{\vector(1,0){6}}}}

\multiput(154,75)(-15,0){10}{\circle*{3}}

{\color{blue}
\multiput(151,71)(10,10){11}{\vector(1,1){9.9}}
\multiput(163,66)(10,-10){10}{\circle*{3}}}

{\color{red}
\multiput(147,79)(10,-10){11}{\vector(1,-1){9.9}}
\multiput(161,84)(10,10){10}{\circle*{3}}}

\end{picture}\end{center}\end{figure}
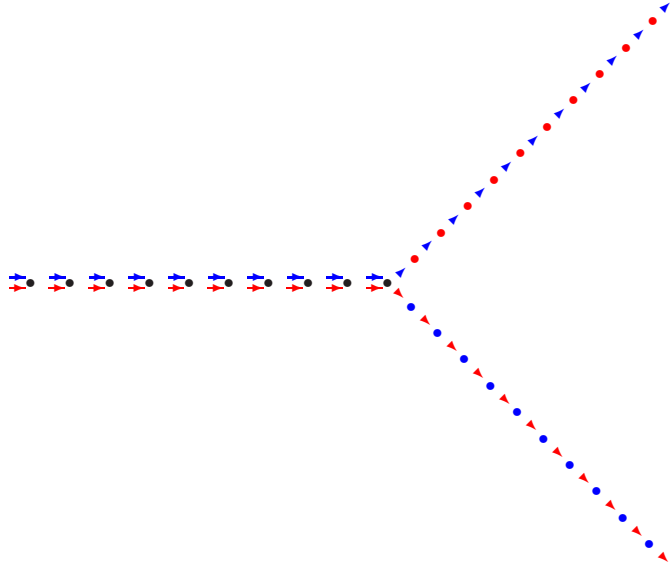

Any such labelled graph defines a finitely generated group $\Gamma=\Gamma(X,E,m) $, subgroup of the full symmetric group of $X$, $\mathbb S(X)$, and  generated by $k$  elements $\alpha_i$, $1\le i\le k$, and their inverses.  By convention, think of an element of $\sigma \in \mathbb S(X)$ as  a rule to move around distinct markers seating above each vertex in $X$. An element $\sigma$ tells us, for each $x$, where to move the marker currently at $x$.  {\em When describing an element $\sigma$, we say that $\sigma$ moves $x$ to $y$ to signify that it moves the marker at $x$ to $y$} (of course, $\sigma$ also moves the label at $y$ to somewhere else).  This is consistent with the fact that we can always describe  a given $\sigma$ by its action on the trivial self-labelling of $X$ by markers in $X$.
 
The action of each  $\alpha_i^{\pm1} $ on the elements of $X$ is given by the labelling  in the sense that $\alpha \cdot x =y $ if and only if  there is an edge $e$ labelled with $\alpha$ and such that $\phi(e)=(x,y)$. In practice, it is often convenient to indicate only the edges labelled by $\alpha_i$, $1\le i\le k$.
For each of these  there is an ``inverse edge'' labelled by the corresponding $\alpha_i^{-1}$ which is omitted. All self loops are also omitted because they can be recovered from the rest of the labelling.  To figure out the action of the product $\alpha_2\alpha_1$ on a vertex $x$, follow the edge at $x$ labelled $\alpha_1$ and from there, follow the edge labelled $\alpha_2$. Proceed similarly for longer products. 

We are interested in a very basic gluing procedure which we now describe. Consider two labelled graphs $(X_i,E_i,m_i)$ as above (with distinct alphabets of possibly different sizes $k_1,k_2$), and subsets $V_1\subset X_1,V_2\subset X_2$ equipped with a bijective map $j:V_1\rightarrow V_2$.
Let $(X,E,m)$ be the the labelled graph of degree $2(k_1+k_2)$ obtained by gluing $(X_1,E_1,m_1)$ and $(X_2,E_2,m_2)$ via the identification of the vertices in $V_1$ with the vertices in $V_2$ (using the bijective map $j$) and adding
appropriate labelled self-loops at all vertices outside $V_1\equiv V_2$.  Obviously, one can glue together more than two labelled graphs along different sets and this can be achieve by repeating the above procedure sequentially.

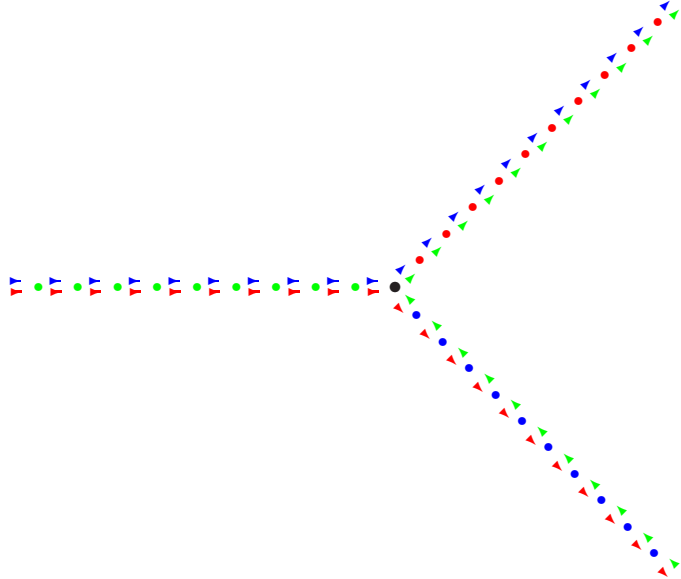
\begin{figure}[h]
\begin{center}\caption{Second step (Houghton group): $t_1$ in blue, $t_2$ in red, $t_3$ in green. Each vertex except the central vertex carry a loop labelled with the generator associated with the given color. }\label{D2} 
\begin{picture}(300,210)(0,-20)  

{\color{blue}{
\multiput(151,77)(-15,0){10}{\vector(1,0){3}}}}
{\color{red}{
\multiput(148,73)(-15,0){10}{\vector(1,0){3}}}}

{\color{green} \multiput(154,75)(-15,0){10}{\circle*{3}}}

\put(154,75){\circle*{4}}

{\color{blue}
\multiput(148,73.5)(10,10){11}{\vector(1,1){9.9}}
\multiput(162,64)(10,-10){10}{\circle*{3}}}

{\color{red}
\multiput(144,75)(10,-10){11}{\vector(1,-1){9.9}}
\multiput(160,85)(10,10){10}{\circle*{3}}}

{\color{green}
\multiput(145,70)(10,10){11}{\vector(1,1){9.9}}
\multiput(161,62)(10,-10){11}{\vector(-1,1){9.9}}}

\end{picture}\end{center}\end{figure}

\begin{example}[Houghton group] See Figures \ref{D1} and \ref{D2}. In this example (it first appeared in \cite{Houghton}, hence the name), we glue together three copies of the Cayley graph of $\mathbb Z$ with canonical generators $t_1,t_2,t_3$. Using the previous description, we start with two copies $\langle t_1\rangle, \langle t_2\rangle$, and identify these two copies 
of $\mathbb Z=\{\dots,-2,-1,0,1,2,\dots\}$  along their respective subsets $\{\dots,-2,-1,0\}$
to obtain an infinite tripod with one branch carrying double edges labelled $t_1,t_2$, one branch carrying simple edges labelled $t_1$ and self-loops labelled $t_2$ and the last branch carrying simple edges labelled $t_2$ and self-loops labelled $t_1$. See Figure \ref{D1} (recall that we only describe one half of the labelling, that is, we omit the description of the ``inverse edges'' labelled $-t_i$). 

Next, we glue a third copy of $\mathbb Z=\langle t_3\rangle$ by identifying  $0t_3$ with the already identified $0t_1=0t_2$,  the points $nt_3$, $n>0$, with $nt_2$,  and the points $nt_3$, $n<0$, with the points $-nt_1$.   See Figure \ref{D2}.

Call $Y$ the vertex set so obtained. The group $\mathcal H_3$ is defined by the labelled graph with vertex set $Y$ and $3$ generators and their marks as described above.  Note again how in figure  \ref{D2} we have omitted all the ``inverse edges'' and that it is a  trivial matter to recover them. In this case, we have actually indicated the existence of self-loops at each vertex by using the color code associated with the generators. But observe that we could have drawn all vertices black instead because the loops can be recovered from the rest of the labelling. 

The group $\mathcal H_3$ can alternatively be described as the group of those permutations of $Y$ which reduce to an eventual translation along each of the ends of $Y$. Indeed, 
call $R_i$ the half-ray on which $t_i$ acts trivially, $i=1,2,3$, and  
orient each of these three half-rays in the direction moving away from $o$.  By recording the far away effect of any element $t$ of $\mathcal H_3$ along each of the rays, we obtain a group homomorphism $ \phi: \mathcal H_3 \ra \mathbb Z^3$ satisfying $$\phi(t_1)= (0,1,-1), \;\phi(t_2)=(1,0,-1), \;\phi(t_3)=(-1,1,0).$$  The image of this map is the subgroup $\Sigma$ of $\mathbb Z^3$ of those elements $n=(n_1,n_2,n_3)$ satisfying $\sum_1^3 n_i=0$ and we have a short exact sequence  $$1\ra \mathbb S_0(Y) \ra \mathcal H_3  \stackrel{\phi}{\ra} \Sigma \ra 1$$
and  $\mathbb S_0(Y)=[\mathcal H_3,\mathcal H_3]$. See, e.g., \cite{SRLee} and the references given therein. 

The careful reader will have noticed that the second step of the construction above was unnecessary. 
The group $\Gamma$ obtain after the first step and generated by $t_1,t_2$ is already the group $\mathcal H_3$.
\end{example}

\begin{example}[Variation on the Houghton group]   Instead of gluing the first two copies of $\mathbb Z$ along the negative integers, let us glue then  along their respective subset $\{-1,0\}$. In a second step, let us glue the third copy of $\mathbb Z =\langle t_3\rangle$ by identifying $0t_3$ with the already identified $ 0t_1\equiv 0t_2$, $-1t_3$ with $1t_1$, and  $1t_3$ with $1t_2 $.  This gives us a graph made of three copies of $\mathbb Z$ glued together along a length $1$ tripod centered at the central point $0$. See  Figure \ref{D3}.  If we call $Y$ the associated vertex set (it has six linear ends) and let $\Gamma$  
be the associated group, we have the short exact sequence
$$1\ra \mathbb S_0(Y) \ra \Gamma  \stackrel{\phi}{\ra} \mathbb Z^3\ra 1$$
and  $[\Gamma,\Gamma]=\mathbb S_0(Y)$. Here $\phi$ associates to any element $\gamma$ of $\Gamma$ the three eventual $\mathbb Z$ translations
observed at infinity along the three pairs of ends of $Y$ associated respectively to $t_1,t_2,$ and $t_3$.

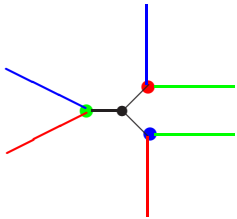
\begin{figure}[h]
\begin{center}\caption{Sketch of the gluing of three copies of $\mathbb Z$ along a length $1$ tripod centered at $0$. }\label{D3} 
\begin{picture}(300,100)(0,25)  

\put(156.5,75){\circle*{4}}
{\color{green} \put(143,75){\circle*{5}} }  
{\color{red} \put(163,84){\circle*{5}}  }
{\color{blue} \put(160.5,66){\circle*{5}}  }
\put(150,75){\line(-1,0){12}}
\put(150,75){\line(1,1){10}}
\put(150,75){\line(1,-1){10}}
{\thicklines
{\color{blue}\put(136,76){\line(-2,1){30}}
\put(159,85){\line(0,1){30}}
}
{\color{red}\put(133,74){\line(-2,-1){30}}
\put(156,65){\line(0,-1){30}}}
{\color{green} \put(155.5,66){\line(1,0){30}}
\put(155.5,84){\line(1,0){30}}}}

\end{picture}\end{center}\end{figure}
\end{example}

\begin{example}[Rooted gluing] \label{exa-R}
Suppose we have two or more labelled graphs $(X_i,E_i,m_i)$, $i=1,2,\dots,\ell$ 
(each with their distinct labellings) and a preferred vertex 
$o_i$. Let $(X,E,m)$ be the labelled graph  with vertex set $X= (\cup_1^\ell X_i\setminus\{o_i\}  )  \cup \{o\}$ corresponding to identifying the points $o_1,o_2,\dots,o_\ell$. One can check that the associated group $\Gamma$ contains a copy of each $\Gamma_i$
where $\Gamma_i$ is the group defined by $(X_i,E_i,m_i)$ and also a copy of  $\mathbb A_0(X)$ which can be  identified as
$$\mathbb A_0(X)
= \langle[ \Gamma_i,\Gamma_j]^\Gamma; i\neq j\rangle.$$
Sometimes, $\Gamma$ will in fact contains the full symmetric group with finite support $\mathbb S_0(X)$, for instance, when one of the $\Gamma_i$ contains an odd  permutation with finite support.

In the case where the labelled graphs $(X_i,E_i,m_i)$ are labelled Cayley graphs of infinite groups $\Gamma_i$, each rooted at the identity element $\id_{\Gamma_i}$, then there is an obvious projection
$\Gamma \ra \Gamma_1\times \cdots \times \Gamma_\ell$ which captures the action of an element $\gamma\in \Gamma$ on each $\Gamma_i$ at infinity and whose kernel is $\mathbb A_0(X)$.  We note that the group $\Gamma$ resulting from  this construction does not depend on the choice of the generators of the groups $\Gamma_i$, $1\le i\le \ell$. 
 \end{example}

\begin{example}[Pocket extension] \label{exa-P}
One of the simplest classes of examples of this type is obtained by joining a rooted labelled graph $(X,E,m)$ with the Cayley graph of the two-element group $\{\id,*\}$ with generator $*$. Let $(X^*,E^*,m^*)$ be the resulting labelled graph. In general, this basic example is already too  complex to be analyzed completely and we will only provide some partial results.

We will however  give sharp general results
in the case when $(X,E,m)$ is the labelled {\em Cayley graph} of an infinite finitely generated group $G$ equiped with a finite generating set $S$. 
In this case, $X^*=G\cup \{*\}$ and the associated group is  $$G_{\circledast}=G\ltimes \mathbb S_0(G\cup\{*\}) .$$
In the case when $G$ is finite, $G_{\circledast} =\mathbb S (X^*)$ and this construction leads to interesting generating sets of the symmetric group. This 
finite case is discussed in \cite{pocketfinite}..
\end{example}

\begin{example}[Star extension of Cayley graphs]\label{star} Compare the following  construction to the pocket extension construction discussed above. Let $(X, E,m)$ be the labelled Cayley graph associated to a group $G$ with finite generating set  $S=\{s_1^{\pm 1},\dots,s_k^{\pm 1}\}$ and labelling alphabet $\mathbf s_1,\dots,\mathbf s_k$.  The simplest way to define the star extension $(G,S)_\star$ of $G$ is as a subgroup of $\mathbb S(G)$ (the group 
of all permutations of the elements of $G$).
Namely, $(G,S)_\star$ is the subgroup of $\mathbb S(G)$ generated by all the left-translation by elements of $G$
and by the $k$ transpositions `` transpose the identity $\id$ and $s_k$'' (which really means, transpose the marker at $\id$ with the marker at $s_k$). To obtain this group as the group generated by a label graph, 
let
$(X_{\star},E_{\star},m_\star)$ be the labelled graph obtained by adding only $4k$ new non-trivial edges to the Cayley graph $(X, E,m)$.  Recall that each edge $e$ is paired with
its ``opposite'' $\check{e}$. In what follows , we omit the description of the ``opposite'' edges so that we only describe $2k$ new edges denoted 
$e^\star_i,f^\star_i$, $1\le i\le k$,
 with
$$x(e^{\star}_i)=\id_G, \;y(e^{\star}_i)=s_i,\; x(f^\star_i)=s_i,\; y(f^\star_i)=\id_G,\; m(e^\star_i)=m(f^\star_i)= \mathbf t_i.$$  At any $x\in X\setminus \{\id_G,s_i\}$, the labelling $\mathbf t_i$  is carried by a self-loop at $x$. 

 As promised, the group $\Gamma=(G,S)_\star$ associated with this labelled graph is the subgroup of $\mathbb S(G)$  generated by the ``translations'' $s_1,\dots,s_k$ and the transpositions $t_i=(\id_G,s_i)$, $1\le i\le k$. 
When $G$ is infinite, $(G,S)_\star$ contains a copy of $G$ 
(translation at infinity) and we have $\Gamma=(G,S)_\star= G\ltimes \mathbb S_0(G) .$  When $G$ is finite, $(G,S)_\star=\mathbb S(G)$.
See Figure \ref{starextZ} for an illustration with $G=\mathbb Z$.

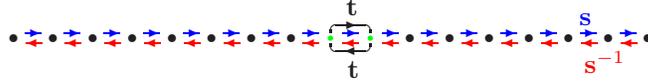
\begin{figure}[h]
\begin{center}\caption{The labelled graph defining the star-extention of $\mathbb Z$; $s^{\pm 1}=\pm 1$, $t=(0,1)$. The green dots mark the elements $0$ and $1$ in $\mathbb Z$. The black dots carry extra loops
associated with the trivial action of $\mathbf t$ and $\mathbf t^{-1}$. The edges between $0$ and $1$ associated with $\mathbf t^{-1}$ are not included in the picture.  \label{starextZ} }
\begin{picture}(200,25)(-50,0)  

{\color{blue}{
\multiput(149,7)(-15,0){16}{\vector(1,0){6}}}}
{\color{red}{
\multiput(152,3)(-15,0){16}{\vector(-1,0){6}}}}
{\color{green} {\multiput(33,5)(15,0){2}{\circle*{2}}}}
\put(40.5,7){\oval(15,6)[t]} \put(41,10){\vector(1,0){2}}
\put(40.5,3){\oval(15,6)[b]}  \put(40,0){\vector(-1,0){2}}
{\multiput(18,5)(-15,0){8}{\circle*{3}}}
{\multiput(63,5)(15,0){7}{\circle*{3}}}
{\color{blue}{\put(126,12){\makebox(0,0){$\mathbf s$}}}}
{\color{red}{\put(130,-4){\makebox(0,0){$\mathbf s^{-1}$}}}}
\put(34.5,17){\makebox(0,0){$\mathbf t$}}
\put(35,-7){\makebox(0,0){$\mathbf t$}}

\end{picture}\end{center}\end{figure}

 \end{example}

\subsection{Isoperimetric profiles}
\label{sec-iso}

Given two functions $f_1,f_2$ taking non-negative real values but defined on an arbitrary 
domain (not necessarily a subset of $\mathbb R$), we write  $f\asymp g$
to signify that there are constants $c_1,c_2\in (0,\infty)$ such that  
$c_1f_1\le f_2\le c_2f_1$.
Given two monotone non-negative real functions $f_1,f_2$, write $f_1\simeq f_2$ 
if there exists $c_i\in (0,\infty)$ such that 
$$c_1f_1(c_2t)\le f_2(t)\le c_3f_1(c_4t)$$ on the domain of definition of 
$f_1,f_2$. Usually, $f_1,f_2$ will be defined on a neighborhood of $0$ or 
infinity and tend to $0$ or infinity at either $0$ or infinity. In some cases, 
one or both functions are defined only on a countable set such as $\mathbb N$. When this is the case, we have to interpret $c_2t,c_4t$ as nearest integers values. 
We denote the associated order by $\lesssim$.
Note that the equivalence relation $\simeq$ distinguishes between power 
functions of different degrees 
and between stretched exponentials $\exp(-t^\alpha)$ 
of different exponent $\alpha>0$ but does not distinguish between 
different rates of exponential growth or decay.  

Given a probability measure $\phi$ on a group $G$, 
let $(S^l_n)_0^\infty$ (resp, $(S^r_n)_0^\infty$)
denotes the trajectory of the left (resp. right) random walk driven by $\phi$ 
(often started at the identity element $\id$).  More precisely, if 
$(X_n)_1^\infty$ are independent identically distributed $G$-valued 
random variables with law $\phi$, then
$$S_n^l= X_n\dots X_1X_0  \;\;(\mbox{ resp. }  S_n^r=X_0X_1\dots X_n).$$
Let $\mathbf P^x_{*,\phi},\; *=l \mbox{ or } r$ be the associated measure 
on $G^\mathbb N$ with $X_0=x$  and 
$\mathbf E^x _{*,\phi}$ the corresponding expectation 
$\mathbf E^x_{*, \phi}(F)=\int_{G^{\mathbf N}}
F(\omega)d\mathbf P^x_{*,\phi}(\omega)$. In particular,
$$\mathbf P^{\idt} _{*,\phi}(S_n=x)=
\mathbf E^{\idt} _{*,\phi}(\mathbf 1_{x}(S_n))=\phi^{(n)}(x).$$
In this work, we find it convenient to work (mostly, but not always) 
with the left version of the 
random walk and we will drop the subscript $l$ in the notation introduced above 
unless we need to emphasize the differences between left and right.
Observe that the random walk on the left is a right-invariant process
since $(X_n\dots X_0)y=X_n\dots (X_0y)$. 
When the measure $\phi$ is symmetric in the sense that 
$\phi(x)=\phi(x^{-1})$ for all $x\in G$, its Dirichlet form is defined by
$$\mathcal E_{\phi}(f,f)=\mathcal E_{G,\phi}(f,f)
=\frac{1}{2}\sum_{x,y\in G}|f(yx)-f(x)|^2\phi(y).$$
This is the Dirichlet form associated with random walk on the left, 
$\mathcal E_\phi^l=\mathcal E_\phi$, and $\mathcal E^r_\phi$ is defined similarly.

The (random walk) group invariant
$\Phi_G$ is a positive decreasing function defined on $[0,\infty)$ 
up to the equivalence relation $\simeq$ which 
describes the probability of return of any random walk on the 
group $G$ driven by a measure $\phi$ that is symmetric, 
has generating support, and a finite second moment with respect 
to a fixed word metric on $G$ (i.e., $\sum_g|g|^2\phi(g)<\infty$). See  \cite{PSCstab}.
Namely, for any finitely generated group $G$ and any measure $\phi$ 
as just described,
$$\forall\, n=1,2,\dots,\;\;\phi^{(2n)}(\id)=\mathbf P_\phi^{\idt}(S_{2n}=\id)
\simeq \Phi_G(n).$$ 

Given a symmetric probability 
measure $\phi$, set
$$\Lambda_{2,G,\phi}(v)=\Lambda_{2,\phi}(v)=\inf\{ 
\lambda_\phi(\Omega): \Omega\subset G,\; |\Omega|\le v\}$$
where
\begin{equation}\label{def-eig}
\lambda_\phi(\Omega)=
\inf\{ \mathcal E_{\phi}(f,f): \mbox{support}(f)\subset \Omega, \|f\|_2=1\}.
\end{equation}
The function $v\mapsto \Lambda_{2,\phi}(v)$ is called the 
$L^2$-isoperimetric profile or spectral profile of $\phi$.

The associated $L^1$-isoperimetric profile is defined by 
\begin{eqnarray*}
\Lambda_{1,G,\phi}(v)&=&\Lambda_{1,\phi}(v)\\
&=&\inf\left\{\frac{1}{2}\sum_{x,y}|f(yx)-f(x)|\phi(y): 
|\mbox{support}(f)|\le v,\;\;\|f\|_1=1\right\}.\end{eqnarray*}
Using an appropriate discrete co-area formula, 
$\Lambda_{1,\phi}$ can equivalently be defined by
$$\Lambda_{1,\phi}(v)= \inf
\left\{|\Omega|^{-1}\sum_{x,y}\mathbf 1_\Omega(x)
\mathbf 1_{G\setminus \Omega}(xy)\phi(y): 
|\Omega|\le v\right\}.$$
If we define the boundary of $\Omega$ to be the set   
$$\partial\Omega=\left\{(x,y)\in G\times G: x\in \Omega, y\in G\setminus \Omega
\right\}$$ and set
$$\phi(\partial \Omega)= \sum_{x\in \Omega,xy\in G\setminus \Omega}\phi(y)$$
then $$\Lambda_{1,\phi}(v)= \inf\{\phi(\partial\Omega)/|\Omega|: |\Omega|\le v\}.$$
It is well-known that
\begin{equation} \label{Cheeger}
\frac{1}{2}\Lambda^2_{1,\phi}\le \Lambda_{2,\phi}\le \Lambda_{1,\phi}.
\end{equation}
Given a non-increasing function $\Lambda$, we define its right-continuous inverse 
$\Lambda^{-1}$ by $$\Lambda^{-1}(s)= \inf\{ v>0:  \Lambda (v)\le s\}.$$
The F\o lner function $\mbox{F\o l}_{G,\phi}
$ is related to the $L^1$-isoperimetric profile defined above by
$$\mbox{F\o l}_{G,\phi}(t)=\inf\{ v: \Lambda_{1,\phi}(v)\le  1/t\}$$
so that $\mbox{F\o l}_{G,\phi}(t)
=\Lambda_{1,\phi}^{-1}(1/t)$ (i.e., $\mbox{F\o l}_{G,\phi}(t)$ is
the right-continuous inverse of the non-decreasing function $\Lambda_{1,\phi}$ at $1/t$ and $t\mapsto \mbox{F\o l}_{G,\phi}(t)$ is left-continuous).
In the literature, the definition $\mbox{F\o l}_{G,\phi}(t)=\inf\{ v: \Lambda_{1,\phi}(v)< 1/t\}$ is sometimes used instead. 

We note that, for $p=1,2$,  the functions $v\mapsto \Lambda_{p,\phi}$ are non-increasing right continuous step-functions changing values only at integer values of their argument $v\in [1,\infty)$.  By definition, $\Lambda_{p,\phi}=\infty$ on $[0,1)$ and 
$$\Lambda_{p,\phi}(v)= s_\phi=\sum_{g\neq \id} \phi(g) \mbox{ for  } v\in [1,2).$$ The right continuous inverse $\Lambda_{p,\phi}^{-1}$ only takes integer values
or the value $+\infty$. It is constant equal to $1$ on $[s_\phi,\infty)$.  More generally, in the definition of $\Lambda^{-1}_{p,\phi}(v)$, the infimum is attained.  Obviously, if $v< \Lambda_{p,\phi}^{-1}(s)$ then $\Lambda_{p,\phi}(v)>s$.

Recall that a finitely generated group $G$ is amenable if and only if (this could be taken as the definition) $\Lambda_{1,G}(v)\simeq 1$ for all $v$.  Equivalently,
$\Lambda_{1,G}^{-1}(s)=+\infty$ for all $s>0$ small enough.

\begin{notation} By elementary comparison arguments, for any two symmetric 
finitely supported probability measures $\phi_1,\phi_2$ with generating support on a group $G$,
we have 
$$\Lambda_{1,G,\phi_1}\simeq \Lambda_{1,G,\phi_2}
\mbox{ and } \Lambda_{2,G,\phi_1}\simeq \Lambda_{2,G,\phi_2}.$$
For this reason we often denote by
$$ \Lambda_{1,G}\;\; (\mbox{ resp. }\; \Lambda_{2,G})$$
the $\simeq$-equivalence class of $\Lambda_{1,G,\phi}$ (resp. $\Lambda_{1,G,\phi}$) with $\phi$ as above. By abuse of notation, we sometimes write $$\Lambda_{p,G}=\Lambda_{p,G,\phi}$$
or understand $\Lambda_{p,G}$ as standing for a fixed representative. 
\end{notation}

\begin{remark} \label{enlarge}
In the definition of $\Lambda_{p,G,\phi}$ (here, $p=1,2$), 
it is not required that $\phi$ generates $G$. In particular, if 
$G_1$ is a subgroup of a group $G_2$ and $\phi$ 
is a symmetric measure supported on $G_1$
then we can consider $\Lambda_{p,G_i,\phi}$ for $i=1,2$. Simple considerations
imply that, in such cases, $\Lambda_{p,G_1,\phi}=\Lambda_{p,G_2,\phi}$. In some instance, 
it might nevertheless be much easier to estimate $\Lambda_{p,G_2,\phi}$ than 
$\Lambda_{p,G_1,\phi}$ directly.  If $\phi$ is finitely supported and
$G_2$ is finitely generated then a simple comparison argument yields 
$\Lambda_{p,G_1,\phi}
\le C(\phi,G_1,G_2) \Lambda_{p,G_2}$.
\end{remark}

We end this section by recalling briefly the fundamental relations that relate the spectral profile $\Lambda_{2,G,\phi}$ to the probability of return
$\mathbf P^{\idt}_\phi(X^r_{2n}=\id)=\phi^{(2n)}(\id)$. 
If $\psi$ is defined as a function of $t$ by $t=\int_1^{1/\psi} \frac{ds}{s\Lambda_{2,\phi}(s)}$ then $\phi^{(2n)}(\id)\lesssim \psi(n)$. 
In the other direction, i.e., for a lower bound on $\Lambda_{2,\phi}$ in terms of $\phi^{(2n)}(\id)$,  see, e.g., \cite[Theorem 2.3]{SCZ-AOP2016}. 
These results are both essentially from \cite{CNash}. For nicely behaved functions, they imply 
that a two-sided estimate of $\Lambda_{2,\phi}$ is equivalent to a two-sided estimate of $\phi^{(2n)}(\id)$.

\section{Rooted gluing of Cayley graphs }  \setcounter{equation}{0} \label{sec-Pocket}

The aim of this section is to prove two complementary theorems which, together, provide matching upper and lower bounds for the $L^1$- and $L^2$-isoperimetric profiles $\Lambda_{1,\Gamma,q}$ and $\Lambda_{2,\Gamma,q}$ for the group $\Gamma$ associated with the rooted gluing (at the identity element) of $\ell$ labelled Cayley graphs $(\Gamma_i,S_i)$, $i=1,\dots,\ell$, equipped with a symmetric measure $q$ of the form
\begin{equation}\label{def-q}
q=\ell^{-1}\sum_1^\ell \mu_i 
\end{equation} where each $\mu_i$ is a symmetric probability measure on $\Gamma_i$ with generating support and each $\Gamma_i$ is viewed as a subgroup 
of $\Gamma$ through the obvious identification.  
The results are expressed in terms of the isoperimetric profiles 
$$\Lambda_{p,i}=\Lambda_{p,\Gamma_i,\mu_i}$$ of the pairs $(\Gamma_i,\mu_i)$, $p=1$ or $2$.   
The measures $\mu_i$ are assumed to be symmetric but they are otherwise arbitrary.

\subsection{Commutator computations} \label{sec-com}
We will need the following lemma.  Recall from Example \ref{exa-R} that $\Gamma$ is defined by its action on 
$$X=\{o\}\cup \left(\cup_1^\ell(\Gamma_i\setminus\{\id_{\Gamma_i}\})\right),$$
and that $\Gamma\leq \mathbb S(X)$ contains a copy of each $\Gamma_i$. The following computations shows that it also contains $\mathbb A_0(X)$.
\begin{lemma} \label{lem-com}
For $g_i\in \Gamma_i$, $g_j\in \Gamma_j$, $i\neq j$, we have
$$[g_i,g_j]= g_ig_jg_i^{-1}g_j^{-1}= (o,g_i,g_j)\in \mathbb A_0(X).$$
If $\sigma \in \mathbb A_0(X)$ has finite support $U$ contained in $ \{o\} \cup (X\setminus \Gamma_j) $ then, for all $g_j\in \Gamma_j$,
$[\sigma,g_j]$ has support in $U \cup \{o,g_j\}$. In fact,
$$[\sigma,g_j]=\left\{\begin{array}{ll} (o,\sigma(o),g_j) &\mbox{ if } \sigma(o)\neq o\\
\idm & { otherwise.}\end{array}\right.$$ 
\end{lemma} 
\begin{proof} The notation $(o,g_i,g_j)$ stands for the element of $\mathbb A_0(X)\subset \Gamma$ which takes  the label at vertex $o$ to $g_i$, the label at vertex $g_i$ to $g_j$ and the label at vertex $g_j $ to $o$. The two computations are done by inspection.
\end{proof}

\begin{lemma} \label{lem-com2}
For $g_i\in \Gamma_i$, $i\in \{r,s,t\}$,  $r\neq s$, we have
$$ g_rg_sg_t=\left\{\begin{array}{ll} g_sg_rg_t (g_t^{-1},g_r^{-1},g_s^{-1})  &\mbox{ if }  r\neq t\\
g_sg_rg_t (g_t^{-1},(g_tg_r)^{-1},g_s^{-1})&\mbox{ if }  r= t. \end{array}\right.$$
\end{lemma}
\begin{proof} Lemma \ref{lem-com} gives us $[g_s^{-1},g_r^{-1}]= (o,g_s^{-1},g_r^{-1}) $ and conjugaison of this cycle by $g_t$ gives the desired result. 
\end{proof}

\begin{lemma} \label{lem-rep} Assume that, for some $0\le \ell'\le \ell$,  the groups $\Gamma_i$ with $i=1,\dots,\ell'$ are infinite, and the groups $\Gamma_i$ with $i=\ell'+1,\dots,\ell$ are finite. 
Then, as a set, the group $\Gamma$ associated with the rooted gluing (at the identity element) of the $\ell$ labelled Cayley graphs $(\Gamma_i,S_i)$
satisfies $$\Gamma= \Gamma_1\times \cdots \times \Gamma_{\ell'} \times \mathbb P_0  \mbox{ (only as sets)},$$ 
where $$\mathbb P_0=\left\{\begin{array}{l} \mathbb S_0(X) \mbox{ if  there  is at least one $\Gamma_i $ which is finite of even order},\\
\mathbb A_0(X) \mbox{ if  none of the $\Gamma_i$ is finite or each finite $\Gamma_i$ has odd order}. \end{array}\right.$$
 Namely, any element $\gamma$ of $\Gamma$ has a unique representation
$$\gamma= \gamma_1\dots\gamma_{\ell'} \tau,\;\;\gamma_i\in \Gamma_i, \;1\le i\le \ell',\;\;\tau\in \mathbb P_0.$$
and all such products appear in $\Gamma$.   In fact, there is a short exact sequence
$$ 1 \ra \mathbb P_0 \ra \Gamma \ra \Gamma_1\times \cdots\times \Gamma_{\ell'}\ra 1.$$
\end{lemma}
\begin{proof} Since $ [\Gamma_i,\Gamma_j]\leq \mathbb A_0(X)$ for $i\neq j$ and that $\Gamma=\langle \Gamma_1,\dots, \Gamma_\ell\rangle $ (viewing each $\Gamma_i$ as a subgroup of $\Gamma$), it is obvious that any element of $\gamma\in \Gamma$ has a representation of the form
$\gamma= \gamma_1\dots\gamma_\ell \tau$ with $\tau\in \mathbb A_0(X)$. Any element $\gamma_i$ that belongs to  a finite $\Gamma _i$ is in $\mathbb S_0(X)$.  This implies that any element of $\gamma\in \Gamma$ has a representation of the form
$\gamma= \gamma_1\dots\gamma_{\ell'} \tau$ with $\tau\in \mathbb P_0$ where $P_0$ is as described above. 
Uniqueness comes from the fact that each of the  $\gamma_i\in \Gamma_i$, $i=1,\dots,\ell'$, is determined uniquely by the action of $\gamma$  on the end $\Gamma_i$ of $X$ at infinity. That any such product does occur follows from the computations in Lemma   
\ref{lem-com} and the fact that the set of all three cycles generate $\mathbb A_0(X)$.
 \end{proof}

\begin{remark} When all the $\Gamma_i$'s are finite, the group $\Gamma$ is finite and equal to either $\mathbb A_0(X)$ or $\mathbb S_0(X)$ with the latter occurring if and only if at least one of the $\Gamma_i$'s has even order. With finite groups, this construction is interesting in so far as it provides a way to construct interesting generating sets for some alternating and symmetric groups.  See \cite{pocketfinite}.
\end{remark} 

\begin{remark} In the short exact sequence described above, the projection onto $\Gamma_1\times \cdots\times \Gamma_{\ell'}$ is given by the action at infinity on each infinite $\Gamma_i$. The sequence does not split when there are more than one infinite $\Gamma_i$ because, although $\Gamma_1$ and $\Gamma_2$ appear in a canonical way as subgroups of $\Gamma$, the  direct  product 
$\Gamma_1\times \Gamma_2$ does not. 
\end{remark}

\begin{remark} What happens if one consider more intricate gluing along some finite subsets of vertices instead of the present rooted gluing at one point?
The overall structure of the groups $\Gamma$ obtained through gluing over finite subsets is roughly the same as that described above. The main possible difference is the exact nature of the subgroup $\mathbb P_0\subset \mathbb S(X)$ that might appear. In most cases, it is possible to show that this subgroup is again either $\mathbb A_0(X)$ or all of $\mathbb S_0(X)$ but some specific configurations may lead to $\mathbb P_0$ being a smaller subgroup of $\mathbb S_0(X)$. In any case, such examples appear to necessitate ad hoc considerations depending of the exact nature of the gluing.  We will not pursue this here but note that, assuming that at least two of the $\Gamma_i$ are infinite (and that the gluing is over finite sets), the group $\mathbb P_0$ always acts transitively on $X$.
\end{remark}

\subsection{Statements of the main results}

Recall that we are given a finite collection of Cayley graphs $(\Gamma_i,S_i)$, $1\le i\le \ell$, each equipped with symmetric probability measure $\mu_i$.
The indexing of these groups is chosen so that the first $\ell'$ of them are infinite and the remaining groups are all finite. We assume throughout that $\ell'\ge 1$, that is, $\Gamma_1$ is infinite.  In this case,  we know describe the isoperimetric and spectral profiles of  the measure $q$ at (\ref{def-q}) on the group $\Gamma$ associated with the rooted gluing (at the identity element) of $\ell$ labelled Cayley graphs $(\Gamma_i,S_i)$, $i=1,\dots,\ell$, in terms of the isoperimetric and spectral profiles  of the pairs $(\Gamma_i,\mu_i)$, $1\le i\le \ell'$.  To simplify notation, we set $\Lambda_{p,i}=\Lambda_{p,\Gamma_i,\mu_i}$, $p=1,2$.

\begin{theorem}[Lower-bound]  \label{th-E} For $p=1,2$ (corresponding respectively to isoperimetric and spectral profile)
and referring to the setup described above,  there are constants $c_1(p), c_2(p)>0$ such that the  isoperimetric profile $\Lambda_{p,\Gamma,q}$  of the symmetric probability measure $q$ defined  on $\Gamma$ at {\em (\ref{def-q})}, satisfies
$$\forall\; v,s>0,\;\;\Lambda_{p,\Gamma,q}(v)\ge \frac{c_1(p)s}{\ell}   \mbox{ for all } v\le \max_{1\le i\le  \ell'}\left\{(c_2(p)\Lambda^{-1}_{p,i}(s))!\right\}.$$
In particular, there exists $v_0$ such that
$$\forall\; v\ge v_0,\;\;\Lambda_{p,\Gamma,q}(v)\gtrsim \max_{1\le i\le \ell'}\left\{\Lambda_{p,\Gamma_i,\mu_i} \left(\frac{\log(1+ v)}{\log(1+\log(1+ v))}\right)\right\}.$$
\end{theorem}

\begin{theorem} [Upper-bound] \label{th-Test1} 
Referring to the setup described above, the isoperimetric profiles $\Lambda_{p,\Gamma,q}$  of $q$ on $\Gamma$, $p=1,2$, satisfies
$$\forall\; v,s>0,\;\;\Lambda_{p,\Gamma,q}(v)\le s   \mbox{ for all } v\ge \left(1+\sum_1^{\ell'} \Lambda^{-1}_{p,i}(s)+\sum_{\ell'<i\le \ell}|\Gamma_i|\right)!.$$
In particular, 
$$\forall\; v> 0,\;\;\Lambda_{p,\Gamma,q}(v)\lesssim \max_{1\le i\le \ell'}\left\{\Lambda_{p,\Gamma_i,\mu_i} \left(\frac{\log(1+ v)}{\log(1+\log(1+ v))}\right)\right\}.$$\end{theorem}

The following statement concerns the special case of the pocket extension of a group $G$. It obviously follows from the previous two results.
\begin{theorem}\label{th-pocket} Let $G$ be an  infinite finitely generated group. Let $G_{\circledast}$ be the pocket extension of $G$. For $p=1,2$ and for all $v>0$,
$$\Lambda_{p,G_{\circledast}}(v)\simeq  \Lambda _{p,G}\left(\frac{\log(1+ v)}{\log(1+\log(1+ v))}\right).$$
\end{theorem}

\subsection{Test functions and proof of the upper-bounds}
This section focuses on the profile upper-bounds stated in Theorem \ref{th-Test1}. We gives the proof for $p=2$ (the case $p=1$ is similar). Recall that, as a set,
$$\Gamma= \Gamma_1\times \cdots \times \Gamma_{\ell'} \times \mathbb P_0.$$

\begin{proof}[Proof of Theorem \ref{th-Test1} for $p=2$].   Fix $s>0$ and $\epsilon>0$.  For each $i\in \{1,\dots,\ell'\}$,  
pick a set $U_i $ and a function $\psi_i$ on $\Gamma_i$ such that
$$ \mathcal E_{\mu_i}(\psi_i,\psi_i)
\le (1+\epsilon) s  \|\psi_i\|_2^2 \mbox{  and } \mbox{ support}(\psi_i)=U_i \mbox{ with } |U_i|\le \Lambda^{-1}_{2,i}(s) .$$

Let $V$ be the set of all elements $\tau$ in $\mathbb P_0\le \Gamma$ with support in
$$\left(\cup_1^{\ell'} U_i^{-1}\right)\cup \left(\cup_{\ell'+1}^{\ell} \Gamma_i\right) $$
where each $U^{-1}_i$ is viewed as a subset of $X$. 

Referring to Lemma \ref{lem-rep}, construct a  test function on $\Gamma= \Gamma_1\times \cdots \times \Gamma_{\ell'} \times \mathbb P_0$ by setting,
for each $\gamma=\gamma_1\dots \gamma_{\ell'} \tau  \in \Gamma$,
$$\psi(\gamma)= \mathbf 1_V(\tau) \prod_1^{\ell'} \psi_i(\gamma_i).$$
Obviously, we have  $$ \|\psi\|_2^2= |V|  \prod_1^{\ell'} \|\psi_i\|_2^2.$$
For any $s_i\in \Gamma_i$ in the support of $\mu_i$, $i\in \{1,\dots,\ell\}$, we compute
$|\psi(s_i\gamma)-\psi(\gamma)|^2$.   Write (uniquely)
$$s_i \gamma= (s_i\gamma)_1\dots(s_i\gamma)_{\ell'} \tau_{s_i\gamma}.$$
Note that, for each $j \in \{1,\dots,\ell'\}$, 
$$(s_i\gamma)_j =\left\{\begin{array}{ll}\gamma_j& \mbox{  if }j\neq i\\
s_i \gamma_i &\mbox{ if } j=i.\end{array}\right.$$
By Lemma \ref{lem-com2},  the permutation $\tau_{s_i \gamma}\in \mathbb P_0$  is the product  of $\tau$ by a permutation supported by 
$$\left\{\begin{array}{cl}\{\gamma_1^{-1}, \dots, (s_i\gamma_i)^{-1},\dots ,\gamma_{\ell'}^{-1}\} & \mbox{ if } \;1\le i\le \ell',\\
\{\gamma_1^{-1}, \dots ,\gamma_{\ell'}^{-1}\}\cup \Gamma _i  &\mbox{ if }\; \ell' <i\le \ell. \end{array}\right.$$  In particular, when $1\le i\le \ell'$, 
\begin{equation} \label{crux}
\tau\in V,\; 
\gamma_j\in U_j, \;1\le j\le \ell', \mbox{ and }s_i\gamma_i\in U_i \mbox{ imply that }
\tau_{s_i\gamma}  \in V, \end{equation} 
and, when $\ell'<i\le \ell$,
\begin{equation} \label{crux'}
\tau\in V,\; 
\gamma_j\in U_j, \;1\le j\le \ell', \mbox{ imply that }
\tau_{s_i\gamma}  \in V. \end{equation} 

Write
$$2\mathcal E_q(\psi,\psi) = \sum_{\psi(\gamma)\psi (s\gamma)\neq 0} |\psi(\gamma)-\psi(s\gamma)|^2\mu(s) + 2\sum_{\psi(\gamma)\neq 0; \psi(s\gamma)=0}|\psi(\gamma)|^2\mu(s)$$
and
\begin{eqnarray*}\lefteqn{
\sum_{\psi(\gamma)\psi (s\gamma)\neq 0} |\psi(\gamma)-\psi(s \gamma)|^2\mu(s)}&&\\
&=&\ell^{-1}\sum_{i=1}^\ell \sum_{ \psi(\gamma)\psi(s_i\gamma)\neq 0} |\psi(s_i\gamma) -\psi(\gamma)|^2\mu_i(s_i)\\
&=&|V|  \ell^{-1}\sum_{i=1}^{\ell'} \prod_{j\neq i} \|\psi_j\|_2^2\sum_{ \psi_i(\gamma_i)\psi_i(s_i\gamma_i)\neq 0}   |\psi_i(s_i\gamma_i) -\psi_i(\gamma)|^2\mu_i(s_i)
\end{eqnarray*}
where the last equality  holds because of (\ref{crux})-(\ref{crux'}). We also have
\begin{eqnarray*}\lefteqn{
\sum_{\psi(\gamma)\neq 0; \psi (s\gamma)=0} |\psi(\gamma)|^2\mu(s)}&&\\
&=&\ell^{-1}\sum_{i=1}^\ell \sum_{ \psi(\gamma)\neq 0 ; \psi(s_i\gamma)= 0} |\psi(\gamma)|^2\mu_i(s_i)\\
&=&|V| \ell^{-1}\sum_{i=1}^{\ell'} \prod_{j\neq i} \|\psi_j\|_2^2\sum_{ \psi_i(\gamma_i)\neq 0; \psi_i(s_i\gamma_i)= 0}   |\psi_i(\gamma)|^2\mu_i(s_i)\end{eqnarray*}
because of (\ref{crux})-(\ref{crux'}).   It follows that
$$\mathcal E_q(\psi,\psi)= |V|\prod_1^\ell \|\psi_i\|_2^2\; \ell^{-1}\sum_{i=1}^{\ell'} \frac{\mathcal E_{\mu_i}(\psi_i,\psi_i)}{\|\psi_i\|_2^2}$$
and
$$\mathcal E_q(\psi,\psi) \le (1+\epsilon) s |V| \,\prod_1^\ell \|\psi_i\|_2^2.$$
After letting $\epsilon$ got to zero, this becomes
$$\frac{\mathcal E_g(\psi,\psi)}{\|\psi\|_2^2} \le s.$$
Because the support of $\psi$  as cardinality at most 
$$\sum_1^{\ell'}|U_i|\times \left( \sum_1^{\ell'}|U_i|+\sum_{\ell'<i\le \ell}|\Gamma_i |\right)!$$ 
and $|U_i|\le \Lambda_{2,i}^{-1}(s)
$, $1\le i\le \ell'$, the desired bound on $\Lambda_{2,\Gamma,q}(v)$ follows.
\end{proof}

\subsection{$L^1$-isoperimetric profile lower bound} \label{sec-beta23}
In this section we consider the basic example obtained by gluing at one point (the neutral element) 
the Cayley graph $(G,S)$ of a finitely generated group $G$ with finite generating set labelled with the alphabet $\mathbf S=\{\mathbf s_1,\dots,\mathbf s_k\}$  and a small finite cycle group $\langle \beta \rangle$ of order $b \in \{2,3\}$. This corresponds to the gluing of two labelled graphs 
$(X_i,E_i,m_i)$ where 
$$X_1=G, \;\;E_1=\{(g, s^{ \pm 1}_i \cdot g ,\mathbf s^{\pm 1}_i): g \in G, i\in \{1,\dots, k\}\}$$ and $m_1((g, s_i^{\pm 1} \cdot g,\mathbf s_i^{\pm 1}))=\mathbf s_i^{\pm 1},$ 
and $$X_2=\mathbb Z/b\mathbb Z, \;\; E_2=\{(x,  \beta^{\pm 1}\cdot x, \boldsymbol \beta^{\pm 1} ): x \in \mathbb Z/b\mathbb Z \}$$ and $m_2((x, \beta ^{\pm 1}\cdot x, \boldsymbol\beta^{\pm}))=\boldsymbol \beta^{\pm 1}$. 
Let $X= \{\beta,\beta^{-1}\}\cup G$ (note that the set notation makes this correct for both $b=2$, in which case $\beta^{-1}=\beta$, and $b=3$). Let 
$\Gamma=\langle \beta,s_1,\dots, s_k\rangle \leq \mathbb S(X)$ be corresponding group.  When necessary, we will use the more explicit notation
$$\Gamma(\beta,G)=\Gamma$$
to describe this abstract construction based on a given group $G$ and  a cyclic group  $<\beta>$ of order $2$ or $3$.
By definition any element $g\in G\leq  \Gamma$ acts on  $G=X\setminus\{\beta^{\pm 1}\}$ by translation on the left,  $(g , g')\ra g g'$ and leaves invariant $\{\beta^{\pm 1}\}$.  The generator $\beta$ acts trivially on $G\setminus \{\id\}$  and rotates cyclically the distinct elements of $\{\id, \beta,\beta^{-1}\}$. 

By  Section \ref{sec-com}, it is clear that $\Gamma= G\ltimes P_0$ where $P_0$ is either $\mathbb S_0(X)$ or $\mathbb A_0(X)$ depending on whether $b=2$ or $3$.

\begin{theorem}  \label{th-beta23}
Referring to the setting described above, let $\mu$ be a symmetric probability measure on $G$ and $\nu$ be the uniform measure on $\langle \beta \rangle$.  There are universal constants $c_1,c_2>0$  (independent of $G$ and $\mu$ and $b=2,3$) such that the symmetric probability $q=\frac{1}{2}(\mu+\nu)$ on $\Gamma$ satisfies
$$\forall\; v,s>0,\;\;\Lambda_{1,\Gamma,q}(v)\ge c_1 s  \mbox{ for all } v\le (c_2 \Lambda^{-1}_{1,G,\mu}(s))!.$$
This holds with $c_1=1/00$ and $c_2=1/32$.  In particular,
$$\Lambda_{1,\Gamma,q} (v)\gtrsim \Lambda_{1,G,\mu} \left(\frac{\log(1+ v)}{\log(1+\log(1+ v))}\right).$$ 
\end{theorem}
The proof of this theorem given below follows closely the argument developed by Anna Erschler to prove her wreath product isoperimetric inequality in \cite{Erschleriso}.  Since $\Gamma= G\ltimes P_0$, we write any element $\gamma\in \Gamma$ as a pair $(g_\gamma,\tau_\gamma)$ where $g_\gamma\in G$ and $\tau_\gamma\in P_0$ so that $\gamma= g_\gamma \tau_\gamma$. The element $g_\gamma$ captures the action of $\gamma$ on $G$ at infinity which is by translation. The element $\tau_\gamma$ is a permutation of $X$ with finite support. Note that for any $s\in G$ and $\gamma=g_\gamma \tau_\gamma\in \Gamma$, we have $s\gamma= (sg_\gamma) \tau_\gamma$, that is, $g_{s\gamma}=sg_\gamma$ and $\tau_{s\gamma}=\tau_\gamma$. Also
$\beta^{\pm1} \gamma= g_\gamma \tau_{\beta^{\pm 1}\gamma}$ with $\tau_{\beta^{\pm1}\gamma}= g^{-1}_\gamma \beta^{\pm 1}g_\gamma \tau_\gamma$.

\begin{definition}  \label{def-KU}
Given a finite subset $U$ of $\Gamma$, set
$$K(U)=\{\tau\in P_0: \tau=\tau_\gamma \mbox{ for some } \gamma\in U\}$$
and let  $EK(U)$ be the set of pairs $\{\tau,\tau'\}\subset K(U)$, $\tau\neq \tau' $, such that  there exists $g\in G$ and  $\epsilon\in \{\pm 1\}$
for which  $(g,\tau)\in U$ and  $(g,\tau')=\beta^{\epsilon }(g,\tau)\in U$ (note that this is indeed a property of the pair $\{\tau,\tau'\}\subset K(U)$).
An element $\tau\in K(U)$ is $a$-satisfactory if
$$\#\{g\in G: (g,\tau)\in U \mbox{ and at least one of }  \beta (g,\tau), \beta^{-1}(g,\tau)\in U\}\ge a.$$
\end{definition}
In words, given the set $U$, an element $\tau\in K(U)$ is $a$-satisfactory is there is at least $a$ locations $g_1,\dots, g_a$  such that, for 
each $i=\{1,2,\dots,a\}$, $(g_i,\tau)\in U$ and $\beta^{\epsilon_i}(g_i,\tau) \in U$ for at least one $\epsilon_i\in \{\pm 1\}$. 

Recall that , by definition, for any finite set $U$,  $q(\partial U)$ is given by
\begin{equation}\label{qboundary}
q(\partial  U)= \frac{1}{2}\sum_{\gamma,s\in \Gamma}| \mathbf 1_U (\gamma)-\mathbf 1_U(s\gamma)|q(s)=
\sum_{\gamma,s\in \Gamma} \mathbf 1_U(\gamma)\mathbf 1_{X\setminus U}(s\gamma) q(s) .
\end{equation}
\begin{lemma}[{Compare \cite[Lemma 2]{Erschleriso}}]  \label{lem-E1}
Let $s\in (0,\epsilon /16)]$, $\epsilon\in (0,1)$.  Assume that  the finite set $U\subset \Gamma$ is such that $q(\partial U) \le s|U|$.   Then we have
$$\#\left\{\gamma\in U: \tau_\gamma \mbox{ is } \frac{1}{2}\Lambda_{1,G,\mu}^{-1}( 4 \epsilon^{-1} s) \mbox{-satisfactory } \right\}\ge 
\left(1-\epsilon\right) |U|.$$ 
\end{lemma}
\begin{proof} Say that an element $\gamma\in U$ is bad if neither $\beta \gamma $ nor $\beta^{-1} \gamma$ is in $U$.  Say $\gamma$ is good if it is not bad.
If $\gamma$ is bad then both $(\gamma,\beta^{\pm 1} \gamma)$ are on the boundary of $U$ and, since $q(\partial U)\le s|U|$ and $q(\{\beta,\beta^{-1}\})\ge 1/4$, we must have
$$\#\{\gamma\in U: \gamma \mbox{ is bad}\}\le  4\, s\, |U| .$$

Let $\mathcal N$ be the set of all $\gamma=(g,\tau)\in U$  such that $\tau$ is non-satisfactory at the level $a= 
\frac{1}{2}\Lambda_{1,G,\mu}^{-1}( 4 \epsilon^{-1} s)$, that is, 
$$\mathcal N= \left\{ \gamma\in U: \tau_\gamma \mbox{ is not } \frac{1}{2}\Lambda_{1,G,\mu}^{-1} ( 4 \epsilon^{-1} s)\mbox{-satisfactory}\right\}.$$
Write $\mathcal N$ as the disjoint union $\mathcal N=\mathcal N_{\mbox{\tiny bad}}\cup \mathcal N_{\mbox{\tiny good}}$.
Suppose the desired conclusion does not hold, that is, 
$$|\mathcal N|> \epsilon |U|.$$
 Since $|\mathcal N_{\mbox{\tiny bad}}|\le  4 s |U|$, we must have
$|\mathcal N_{\mbox{\tiny bad}}|< 6 \epsilon^{-1} \, s\, |\mathcal N|$.  For $\tau \in \mathcal P_0$, write
\begin{eqnarray*}
\mathcal N (\tau)&=&\{(g,\sigma)\in \mathcal N: \sigma=\tau\}\\
\mathcal N_{\bullet}(\tau)&=& \{(g,\sigma)\in  \mathcal N_{\bullet}: \sigma=\tau\},\;\bullet=\mbox{ good or  bad}. 
\end{eqnarray*}
Note that  
$$\mathcal N_\bullet= \bigcup_{\tau\in P_0} \mathcal N_\bullet(\tau)$$
and let  $\mathcal C$ be the set of all permutations in $P_0$ such that
$$ |\mathcal N(\tau) | \le  2|\mathcal N_{\mbox{\tiny good}}(\tau)|    \;\;(\mbox{i.e., } |\mathcal N_{\mbox{\tiny bad}}(\tau)|\le |\mathcal N_{\mbox{\tiny good}}(\tau)| )$$
Observe that
\begin{eqnarray*}|\mathcal N| &= & \sum_{\tau \in \mathcal C}  |\mathcal N (\tau)|+\sum_{\tau \not\in \mathcal C}  |\mathcal N (\tau)|
\\
&\le & \sum_{\tau \in \mathcal C}  |\mathcal N(\tau)|+2 \sum_{\tau \not\in \mathcal C}  |\mathcal N_{\mbox{\tiny bad}}(\tau)| 
\\
&\le &  \sum_{\tau \in \mathcal C}  |\mathcal N(\tau)|+2 |\mathcal N_{\mbox{\tiny bad}}|  
\end{eqnarray*}
Since $|\mathcal N_{\mbox{\tiny bad}}| <  4 \epsilon^{-1} s |\mathcal N| \le  \frac{1}{4} |\mathcal N|  $,  it follows that
\begin{equation} \label{NC}
|\mathcal N| \le
2 \sum_{\tau\in \mathcal C}|\mathcal N(\tau)|.\end{equation}
We now estimate from below the size of the boundary of $U$. For this purpose, set $\overline{\mathcal N}(\tau)=\{g\in G: (g,\tau)\in \mathcal N(\tau)\}$
and define $\overline{ \mathcal N}_\bullet(\tau)$ in the same fashion for $\bullet= \mbox{ good or  bad}$ . Obviously
$|\overline{\mathcal N}_\bullet(\tau)|=|\mathcal N_{\bullet}(\tau)|$ where $\bullet$ is blank, good or bad.  Since $q=\frac{1}{2}(\nu+\mu)$ with $\mu$ supported on $G$, we have (see (\ref{qboundary}))
\begin{eqnarray}
2q(\partial U)&\ge & \sum_{\gamma\in U, s\in G} \mathbf 1_{\Gamma \setminus U}(s\gamma) \mu(s)
\ge   \sum_{\gamma\in \mathcal N, s\in G} \mathbf 1_{\Gamma \setminus U}(s\gamma)\mu(s) \nonumber \\
&=& \sum_{\tau\in P_0} \sum _{g\in \overline{\mathcal N}(\tau), s\in G} \mathbf 1_{G\setminus \overline{\mathcal N}(\tau)} (sg) \mu(s)\nonumber \\
&\ge  &\sum_{\tau\in \mathcal C} \mu (\partial \overline{\mathcal N}(\tau)). \label{boundary}
\end{eqnarray}
If $g\in \overline{\mathcal N}_{\mbox{\tiny good}}(\tau)$ then
at least one of  $\beta^{\pm 1}(g,\tau)$ is in $U$ (see the definition or bad/good) and  $\tau$ is not $\frac{1}{2}\Lambda_{1,G,\mu}( 4\epsilon^{-1} s)$-satisfactory. 
Hence, we must have 
$$|\overline{\mathcal N}_{\mbox{\tiny good}}(\tau)| < \frac{1}{2} \Lambda_{1,G,\mu}^{-1}(4 \epsilon^{-1} s).$$
When $\tau\in \mathcal C$, it follows that the set $\overline{\mathcal N}(\tau) \subset G$ has size bounded by 
$$|\overline{\mathcal N}(\tau)|\le 2 |\overline{\mathcal N}_{\mbox{\tiny good}}(\tau)| < \Lambda_{1,G,\mu}^{-1}( 4 \epsilon^{-1} s ).$$
This implies
$$ \mu (\partial \overline{\mathcal N}(\tau)) \ge  4 \epsilon^{-1}  s   | 
\overline{\mathcal N}(\tau)|. $$
Using this inequality  in (\ref{boundary}),  it follows that
$$
q(\partial U) \ge  2 \epsilon^{-1} s  \sum_{\tau\in \mathcal C}  | \overline{\mathcal N}(\tau)| \ge    \epsilon^{-1} s  |\mathcal N| 
>    \, s \, |U|.
$$
where the last  inequality follows from the assumption that $|\mathcal N|>\epsilon |U|$. This contradicts the main hypothesis. Hence it must be the case that
 $|\mathcal N| \le \epsilon |U|$, that is,
  $$\#\left\{\gamma\in U: \tau_\gamma \mbox{ is $\frac{1}{2}$}\Lambda_{1,G,\mu}^{-1}( 4 \epsilon^{-1} s) \mbox{-satisfactory } \right\}\ge 
\left(1-\epsilon\right) |U|.$$ 
\end{proof} 

The next lemma is a version of the {\em edge removal lemma} of A. Erschler \cite[Lemma 1]{Erschleriso}. We need to generalize the notion of $a$-satisfactory vertex.  Given the graph $(K(U),EK(U))$  (recall that $K(U)\subset P_0$ is a finite subset of permutations), consider a subgraph $(K',E')$ of $(K(U),EK(U))$.
A vertex $\tau\in K'$ is $a$-satisfactory in $(K',E')$ is there are at least $a$ distinct elements $g\in G$ such that  $\gamma= (g,\tau) \in U$ and at least one of
$(\tau,\tau_{g,+}),(\tau,\tau_{g,-})\in E'$ where $\tau_{g,+},\tau_{g,-} $ are defined  by  
$\beta^{\pm 1}(g,\tau) =(g,\tau_{g,\pm})$, that is
$$ \tau_{g,\pm}= g^{-1} \beta^{\pm 1}g \tau . $$
Note that  $\tau_{g, \epsilon}=\tau_{g',\eta}$ if and only if $g=g'$ and $\epsilon=\eta$. Thus, if $\tau$ is $a$-satisfactory,  there are at least $a$ 
distinct edges adjacent to $\tau$ in $(K',E')$.  If $\tau$ is not $a$-satisfactory in $(K',E')$ then there are less than $2a$ edges adjacent to $\tau$ in $(K',E')$. 

We say that an edge $\{\tau,\tau'\}\in E'$ is $a$-satisfactory if  both of its ends, $\tau,\tau'$ are $a$-satisfactory.  Let $\mbox{NS}(K',E',a)$ be the set of all non-$a$-satisfactory edges for $(K',E')$.   

\begin{lemma}  \label{lem-E2}
Assume that 
$$\frac{|\mbox{\em NS}(K(U),EK(U),a)|}{|EK(U)|}\le \frac{1}{4}.$$
Then there exists a subgraph $(K',E')$, $K'\neq \emptyset$, all of whose vertices are $(a/4)$-satisfactory.
\end{lemma}
\begin{proof} Set  $K_0=K(U)$, $E_0=EK(U)$. Consider the vertices in $(K_0,E_0)$,  which are not $(a/4)$-satisfactory in $(K_0,E_0)$. 
Remove these vertices and all their adjacent edges to obtain $(K_1,E_1)$. If some of the vertices in $(K_1,E_1)$ are not $(a/4)$-satisfactory in $(K_1,E_1)$, remove them and all adjacent edges and repeat.  labelled each vertex $\tau$ with the time $i_\tau$ of its removal and orient each of the edges removed towards the vertex that remains after the removal of the edge (if both ends of the edge are removed at the same time, orient the edge arbitrarily).  
Let $\mathcal R^i=\{\tau\in K_0: i_\tau=i\}$ be the set of all vertices removed  at time $i$. By definition, such a vertex is in $K_{i-1}$ but not in $K_i$.  
For each vertex $\tau\in K_0$,  record the two sequences of numbers
$$  a_{\tau,j},\;j\le i_\tau \mbox{ and }\; b_{\tau,j},\;j \ge i_\tau$$
where $a_{\tau_j}$ is the number of oriented edges $(\tau',\tau)$ removed at time $j\le i_\tau$ 
and $b_{\tau_j}$ is the number of oriented edges $(\tau,\tau')$ removed at time $j\ge i_\tau$ (in both cases, $j=i_{\tau'}$). 

By definition we have that the total number $T$ of removed edges in the
whole process is 
\[
T=\sum_{i=1}^{\infty}\sum_{\tau\in\mathcal{R}^{i}}\sum_{1\le j\le i}a_{j,\tau}=\sum_{i=1}^{\infty}\sum_{\tau\in\mathcal{R}^{i}}\sum_{j\ge i}b_{\tau,j}.
\]
To show that the process must end with a non-empty graph, we argue by contradiction. Assume instead that the removal process ends with the empty graph (every vertex gets removed at some point). Since
every vertex gets removed in the end, we have that 
$$\sum_{1\le j\le i_\tau}a_{j,\tau}+\sum_{j\ge i_\tau}b_{\tau,j}$$
is exactly the degree of the vertex $\tau$ in $(K(U),EK(U))=(K_0,E_0)$. 

Write $\mathcal{N}_{a}=\mathcal{N}(K_0,E_0,a)$ for the set of vertices that
are non-$a$-satisfactory in $(K_0,E_0)$ and split the sum for $T$ into
\[
T=\sum_{i=1}^{\infty}\sum_{\tau\in\mathcal{R}^{i}\cap\mathcal{N}_{a}}\sum_{j\ge i}b_{\tau,j}+\sum_{i=1}^{\infty}\sum_{\tau\in\mathcal{R}^{i}\cap(K\setminus\mathcal{N}_{a})}\sum_{j\ge i}b_{\tau,j}.
\]
In the first summation, since $\tau$ is non-$a$-satisfactory, 
the sum is bounded by the total number of non-$a$-satisfactory edges
\[
\sum_{i=1}^{\infty}\sum_{\tau\in\mathcal{R}^{i}\cap\mathcal{N}_{a}}\sum_{j\ge i}b_{\tau,j}\le|NS(K_0,E_0,a)|.
\]
Now we bound the second sum. By definition, the vertices removed during the first round are non-$a/10$-satisfactory in $(K_0,E_0)$.  It follows that
the second sum actually starts from $i=2$.
From the edge removal procedure, $\tau\in\mathcal{R}^{i}\cap(K\setminus\mathcal{N}_{a})$,
$i\ge 2$, implies that  $\tau$ was $a$-satisfactory in $(K_0,E_0)$ but has becomes non-$a/4$-satisfactory in $(K_{i-1},E_{i-1})$ and gets removed in round $i=i_\tau$.
Therefore 
\[
\deg\tau=\sum_{1\le j\le i_\tau}a_{j,\tau}+\sum_{j\ge i_\tau}b_{\tau,j}\ge a\mbox{ and }\sum_{j\ge i_\tau}b_{\tau,j}\le 2a/4=a/2.
\]
It follows that for any $\tau \in\mathcal{R}^{i}\cap(K\setminus\mathcal{N}_{a})$,
\[
\sum_{1\le j\le i_\tau}a_{j,\tau}\ge  2 \sum_{j\ge i_\tau}b_{\tau,j}.
\]
Summing up, we have
\[
\sum_{i=1}^{\infty}\sum_{\tau\in\mathcal{R}^{i}\cap(K\setminus\mathcal{N}_{a})}\sum_{j\ge i}b_{\tau,j}\le\frac{1}{2}\sum_{i=2}^{\infty}\sum_{\tau\in\mathcal{R}^{i}\cap(K\setminus\mathcal{N}_{a})}\sum_{1\le j\le i}a_{j,\tau}\le\frac{1}{2}T.
\]
Combining the two estimates, it follows that 
\[
T\le |NS(K,E,a)|+\frac{1}{2}T
\]
and, because $|NS(K,E,a)|\le\frac{1}{4}|EK|$,
\[
T\le 2 |NS(K,E,a)|\le\frac{1}{2}|E(K)|.
\]
This contradicts the assumption that the process ends with the empty graph. 
\end{proof}
\begin{lemma}\label{lem-E3}  Fix $U\subset G$ and let $(K,E)$ be a subgraph of  $(K(U),EK(U))$ such that each vertex in $K$ has at least $2b$ distinct 
neighbors in $(K,E)$. Then
$$|K|\ge b!.$$
\end{lemma}
\begin{proof} We proceed by induction on $b$.  The statement is obviously true for $b=1$.  Suppose it is true for $b=k-1$.  Let $\tau_0\in K$. By assumption there exists $k$ distinct elements $g_1,\dots, g_{k} \in G$ and $\epsilon_i\in \{\pm 1\}$ such that  
$$\tau_i= (g_i^{-1}\beta^{\epsilon_i}g_i)  \tau_0\in K \mbox{ and } \epsilon_1=\epsilon_2=\dots=\epsilon_k=\epsilon_0\in \{\pm 1\}.$$
(here we can assume that the $\epsilon_i$s are all the same because of the assumption 
that $\tau_0$ has $2k$ neighbors).

Let $x_{0}= \tau_0^{-1} (\beta^{-\epsilon_0})$. Then, by construction,   $\tau_i ( x_0)=g_i^{-1} $. 
For each $i\in \{i,\dots,k\}$ consider the set of vertices $\mathcal P_i$ which is the connected component of $\tau_i$ in the subgraph $\Delta_i$ of $(K,E)$
obtained by removing all edges labelled $g_i^{-1}\beta^{\epsilon} g_i$, $\epsilon\in \{\pm 1\}$.   Each vertex in $\Delta_i$ has at least $2(k-1)$ neighbors in $\Delta_i$ so that,
by the induction hypothesis, $|\mathcal P_i|\ge (k-1)!$. It remains to check that
the sets $\mathcal P_i$, $1\le i\le k$ are disjoints. This is the case because, by inspection
of the definitions, for each
$\sigma\in \mathcal P_i$, we have $\sigma(x_0)=g_i^{-1}$.
\end{proof}

\begin{proof}[Proof of Theorem \ref{th-beta23}] Let $\Gamma=\langle \beta,s_1,\dots,s_k\rangle =G\ltimes P_0$ be as in Theorem \ref{th-beta23}. Let
$\mu$ be a symmetric probability measure on $G$ and $\nu$ be the uniform measure on $\langle \beta\rangle $ (recall that $\beta$ has order $2$ or $3$).
Let $q=\frac{1}{2}(\mu+\nu)$.  Let $U$ be a finite subset of $\Gamma$ such that
$q(\partial U) \le   (s/ 84)|U| $ with $s\le 1/4$.  By Lemma  \ref{lem-E1} with $\epsilon=1/16$, we have
$$\#\{g\in U: \tau_g \mbox{ is } \frac{1}{2}\Lambda_{1,G,\mu}^{-1}(s)\mbox{-satisfactory }\} \ge \left(1-\frac{1}{16}\right)|U|. $$
It follows that the subgraph  $(K(U),EK(U)$ from Definition \ref{def-KU} satisfies
$$\frac{|NS(K(U),EK(U),a)|}{|EK(U)|}\le \frac{1}{4} \mbox{ for } a=\frac{1}{2}\Lambda_{1,G,\mu}^{-1}(s).$$
It follows from Lemma  \ref{lem-E2} that there is a non-empty subgraph $(K',E')$
all of whose vertex has at least  $a/4$ neighbors. By Lemma \ref{lem-E3},
$$|K(U)\ge |K'|\ge \lfloor a/8\rfloor ! .$$
Obviously,  $|U|\ge  |K(U)|$. Hence, we
have proved that  for any $s\in (0,1/4)$ and any finite subset $U\subset \Gamma$ with $q(\partial U)\le (s/84)|U|$, we must have 
$$|U|\ge   \lfloor \frac{1}{16}\Lambda^{-1}_{1,G,\mu}(s)\rfloor !.$$
This completes the proof of Theorem \ref{th-beta23}.
\end{proof}

 \subsection{Proof of Theorem \ref{th-E} for $p=1$ }
 Theorem \ref{th-E} describes lower bounds on the isoperimetric ($p=1$) and spectral profiles ($p=2$) of any group  $\Gamma$ obtained from the Cayley graphs 
 $(\Gamma_i,S_i)_1^\ell$ of $\ell$ finitely generated groups $\Gamma_i$ via rooted gluing at the identity element. Each group $\Gamma_i$ is equipped with a symmetric measure $\mu_i$ and the  group $\Gamma$ is equipped with the associated  symmetric measure $q$ defined at (\ref{def-q}). 
 
 Suppose that $i\in \{1,\dots,\ell\}$ is such that at least one of the $\Gamma_j, j\neq i$ has at least 3 elements. Then,
 according to the commutator computations recorded in Lemmas \ref{lem-com}-\ref{lem-com2}, the group $\Gamma$ contains 
 a group $\Gamma^*_i = \langle S_i,\beta\rangle =\Gamma_i\ltimes \mathbb A_0(X_i^*) $ where $X_i^* =\Gamma_i\cup \{\beta,\beta^{-1}\}$ 
and  $\beta $ also stands for the three cycle $(\id,\beta,\beta^{-1})$ (see Section \ref{sec-beta23}). If every $\Gamma_j$, $j\neq i$, is a two element group then, obviously, $\Gamma$ contains a subgroup $\Gamma^*_i = \langle S_i,\beta\rangle =\Gamma_i\ltimes \mathbb S_0(X_i^*) $ where $X_i^* =\Gamma_i\cup
 \{\beta\}$  and  $\beta $ also stands for the transposition $(\id,\beta)$.   In both cases, let $\nu$ be the uniform measure on $\langle \beta\rangle$ and set
 $\mu_{i*}=\frac{1}{2}(\mu_i+\nu)$.  By a simple comparison argument, there  is a positive constant $c$ which depends only on a positive lower bound on 
 $$\mu_*=\inf \left\{\mu_i(s): s\in S_i, i\in \{1,\dots,\ell\}\right\},$$ 
 such that 
 $$\Lambda_{p,\Gamma,q}(v) \ge  c \ell^{-1}  \Lambda_{p,\Gamma_i^*,\mu_{i*}}(v).$$
 Hence, in the case $p=1$ (isoperimetric profile) the conclusion of Theorem \ref{th-E} follows from Theorem \ref{th-beta23}.
 
 \subsection{Proof of Theorem \ref{th-E} for $p=2$ } 
 By the same comparison technique used above in the case $p=1$, in order to prove the spectral profile statement (i.e., the case $p=2$) of Theorem \ref{th-E}, it suffices to prove the spectral profile version of Theorem \ref{th-beta23} which is the following statement.
 \begin{theorem}  \label{th-beta23(2)}
Referring to the setting  of {\em Theorem \ref{th-beta23}}, let $\mu$ be a symmetric probability measure on $G$ and $\nu$ be the uniform measure on $\langle \beta \rangle$.  There are universal constants $a_1,a_2>0$  (independent of $G$ and $\mu$ and $b=2,3$) such that the symmetric probability $q=\frac{1}{2}(\mu+\nu)$ on $\Gamma$ satisfies
$$\forall\; v,s>0,\;\;\Lambda_{2,\Gamma,q}(v)\ge a_1 s  \mbox{ for all } v\le (a_2 \Lambda^{-1}_{2,G,\mu}(s))!.$$
 In particular, for all $v>0$,
$$\Lambda_{2,\Gamma,q} (v)\gtrsim \Lambda_{2,G,\mu} \left(\frac{\log(1+ v)}{\log(1+\log(1+ v))}\right).$$ \end{theorem}
 \begin{proof}We adapt the technique of \cite[Section 4]{SCZ-AOP2016} which involves comparison with well chosen spread-out measures.
 By \cite[Theorem 4.7]{SCZ-AOP2016} (with $\phi=\mu$ and $\alpha=1/2$), for any $v\ge 1$, we can associate to the symmetric probability measure $\mu$ on $G$ another symmetric probability measure on $G$, $\zeta_{\mu,v}=\zeta_v$ such that  (the constant $c$ below is a positive numerical constant independent of $v,\mu, G$)
 $$\Lambda_{1,G,\zeta_v}(v)\ge 1/2  \mbox{ and }\; \mathcal E_\mu\ge c\Lambda_{2,G,\mu}(8v) \mathcal E_{\zeta_v}.$$
 By Theorem \ref{th-beta23}, the measure $q_v=\frac{1}{2}(\nu+\zeta_v)$ on the group $\Gamma$ satisfy
 $$\Lambda_{1, \Gamma,q_v} (u)\ge c_1/2  \mbox{ for }   u\le  (c_2 v )!.$$
 Using the left-hand side of (\ref{Cheeger}), this also gives
 $$\Lambda_{2, \Gamma,q_v} (u)\ge c^2_1/8  \mbox{ for }   u\le  (c_2 v )!.$$ 
 But it is clear that we also have  (recall that $q=\frac{1}{2}(\nu+\mu)$ on $\Gamma$)
 $$  \mathcal E_q \ge c\Lambda_{2,G,\mu}(8v) \mathcal E_{q_v}.$$ 
 So, for any $v\ge 1$, we have
 $$\Lambda_{2,q,\Gamma}(u)\ge c (c_1^2/8) \Lambda_{2,G,\mu}(8v)  \mbox{ for }   u\le  (c_2 v )!.$$  
 Setting $s=\Lambda_{2,G,\mu}(8v)$, this reads
 $$\Lambda_{2,q,\Gamma}(u)\ge c (c_1^2/8) s \mbox{ for }   u\le  \left(\frac{c_2}{8} \Lambda^{-1}_{2,G,\mu}(s) \right)!.$$  
 \end{proof}
 
\section{Houghton groups and variations}

Let $Y_k =\{o\}\cup (\cup_1^k R_i) $ where each $R_i$ is a copy of $\{1,2,\dots\}$, explicitly, 
$R_i=\{r_{i,m}: m=1,2\dots\}$. In words, $Y_k$ is the union of  $k$ copies of the   
non-negative integers where all copies of $0$ have been identified. The Houghton group $\mathcal H_k$ is the group of all permutation of $Y_k$ which are eventual translations on each ray $R_i$, $1\le i\le k$.  By definition, this means that there is a projection $\phi: \mathcal H_k \ra \mathbb Z^k$ which associates  to each element $h$ of $\mathcal H_k$, $\phi(h)=(m_1,\dots,m_k)$ where $m_i$ captures the (positive or negative) amount of eventual  translation away from $0$ along the ray $R_i$. By definition, the kernel of $\phi$ is contained in the subgroup $\mathbb S_0(Y_k)$ of those permutations
that have finite support and it must be all of them. The image of $\phi$ is the subgroup $\Sigma=\{(m_i)_1^k: \sum_1^k m _i=0\}$ of $\mathbb Z^k$.  
Indeed, by inspecting the action of an element $g\in \mathcal H_k$ on the star 
$$\mathbf S(N)=\{o\}\cup_1^k\{r_{i,n}: 1\le n\le N\}$$ where $N$ is chosen so large that $g$ acts by translation on each of  $\{r_{i,N+1},r_{i,N+2},\dots\}$, one sees that $\phi(g)=(m_1,\dots,m_k)\in \Sigma$, that is, $\sum_1^km_i=0$.
Also, for any pair $(i,j)$, $1\le i<j\le k$, consider the element $h_{i,j}$ of $\mathcal H_k$ which is``translation by $1$" along the copy of $\mathbb Z$ obtained by setting $0=o$, $-n=r_{i,n}, n=r_{j,n}, n=1,2,\dots$. Clearly, the images $\phi(h_{i,j})$, $1\le i<j\le k$, generates $\Sigma$. 
It is plain to check that, for  $i_1<j_1\le k$, $i_2<j_2\le k$, $j_1\neq j_2$, the commutator $[h_{i_1,j_1},h_{i_2,j_2}]$ is the transposition $(r_{j_1,1},r_{j_2,1})$ when $i_1=i_2$ and the three cycle $(o,r_{j_1,1},r_{j_2,1})$ if $i_1\ne i_2$. It easily follows  that  the elements $h_{i,j}$, $1\le i<j\le k$ 
generate $\mathcal H_k$ (in fact we only need $k-1$ of then chosen so that each ray $R_i$ is represented at least once)  and that
we have a short exact sequence  $$1\ra \mathbb S_0(Y) \ra \mathcal H_k \stackrel{\phi}{\ra} \Sigma \ra 1$$
with, in addition,  $\mathbb S_0(Y)=[\mathcal H_k,\mathcal H_k]$.  See, e.g., \cite{SRLee} for details and earlier references.

Given a family  $\mathfrak p$ of $ k\choose 2$ symmetric probability measures $p_{i,j}$, $1\le i<j\le k$, on $\mathbb Z$,  define a symmetric probability measure $q_{\mathfrak p}$
on $\mathcal H_k$ by setting
\begin{equation}\label{Houghtonq}
q_{\mathfrak p}(g) ={k\choose 2} ^{-1}\sum_{1\le i<j\le k} \sum_{n\in \mathbb Z} p_{i,j}(n)\mathbf 1_{\{h_{i,j} ^n\}}(g).\end{equation}
This probability measure is supported on the powers of the generators $h_{i,j}$ and we allow the possibility that $p_{i,j}(0)=1$ (at least for some pairs $(i,j)$).

\begin{theorem}\label{th-Houghton-1}
Referring to the setting and notation introduced above, let  $$(i_0,j_0)\neq (i_1,j_1),\;\;  1\le i_0<j_0\le k,\;1\le i_1<j_1\le k, $$
be such that $p_{i_m,j_m}(1)>0$
 for $m=0,1$. Then 
 there are positive constants $c_1(k,\mathfrak p),c_2(k,\mathfrak p)$ such that, for all $v,s>0$, the profiles $\Lambda_{p,\mathcal H_k,q_{\mathfrak p}}$  of the symmetric probability measure $q_{\mathfrak p}$ on $\mathcal H_k$, $p=1,2$, satisfies
$$\Lambda_{p,\mathcal H_k,q_{\mathfrak p}}(v)\ge   c_1(k,\mathfrak p) s   \mbox{ for all } v\le \left( c_2(k,\mathfrak p)
\max_{m=0,1} 
  \left\{\Lambda^{-1}_{p,\mathbb Z, p_{i_m,j_m}}(s)\right\}\right)!.$$
\end{theorem}
\begin{proof}  Consider for simplicity the case when  $i_0\ne i_1$, $j_0\neq j_1$. 
In $\mathcal H_k$, consider the subgroups $Z_m=<h_{i_m,j_m}>$ and the three cycles $\beta_m=(o,r_{1,i_{m'}},r_{1,j_{m'}})$ where $m'=m+1 \mod 2$, $m
\in \{0,1\}$.  By construction,  $\mathcal H_k$ contains a copy of  $\Gamma(\beta, \mathbb Z)$ with $\mathbb Z=<h_{i_m,j_m}>$ and $\beta=\beta_m$.
Further, let  $\mu_m$ denote the measure $p_{i_m,j_m}$ on $\mathbb Z=<h_{i_m,j_m}>$. Let $\nu_m$ be the uniform measure on $<\beta_m>$
and set $q=\frac{1}{2}(\nu_m+\mu_m)$, then a simple comparison argument implies that
$$\Lambda_{p,\mathcal H_k,q_{\mathfrak p}} \ge c\Lambda_{p,\mathcal H_k,q_m}, \;\;p=1,2.$$
Here the exact value of the positive constant $c$ depends on lower bounds on $p_{i_m,j_m}(1)$, $m=1,2$.
Theorem \ref{th-beta23} (and its spectral version Theorem \ref{th-beta23(2)}) implies
$$\Lambda_{p,\mathcal H_k,q_{\mathfrak p}}(v)\ge c_1 s  \mbox{ for all } v\le (c_2 \Lambda^{-1}_{p, \mathbb Z, p_{i_m,j_m}}(s))! .$$

\end{proof}

The next theorem provide matching upper-bound for the isoperimetric and spectral profiles under certain assumptions.    To obtain this upper-bound, we 
follow a line of reasoning that is similar to the one used for the rooted gluing of Cayley graphs. However, there are some significant differences in some of the details.

\begin{lemma}  \label{lem-HRep}
Exclude one of the rays, say $R_k$. For each remaining $R_i$, $1\le i\le k-1$, set $g_i=h_{i,k}$ and $Z_i=\langle  g_i\rangle$.
Any element $ \gamma \in \mathcal H_k$ admits a unique decomposition of the form
$$g=z_1\dots z_{k-1} \tau  \mbox{ with }  z_i\in Z_i , \;1\le i\le k-1, \mbox{ and  } \tau\in \mathbb S_0(Y).$$
\end{lemma}
\begin{proof}  Since $\mathcal H_k$ is generated by $g_1,\dots, g_{k-1}$ (see above) and have commutators in $\mathcal S_0(Y)$, it is plain that 
every element $g\in \mathcal H_k$ can be written as described above. To prove uniqueness, we observe that the  integer vector $(z_1,\dots,z_{k-1})$
is uniquely determined by the condition that  $\phi(g)=(-z_1,-z_2,\dots,-z_{k-1}, \sum_1^{k-1}z_k)$ 
\end{proof}

\begin{theorem}  \label{th-Houghton-2}
Fix $p=1,2$. For any $s\in (0,1]$, let $\psi_s$ be a symmetric non-negative function on $\mathbb Z$  supported on $(-r(s),r(s))$, normalized by  $\|\psi_r\|_p=1$, and such that for any $1\le i<j\le k$, 
$$\frac{1}{2}\sum_{x,y\in \mathbb Z}|\psi_s(x+y)-\psi_r(x)|^pp_{i,j}(y)\le  s .$$
Then
$$\forall v,s>0, \;\; \Lambda_{p,\mathcal H_k,q}(s)\le 2s \;\mbox{ for all } \;  v\ge ([k+3]r(s))!$$
\end{theorem}
\begin{proof} The case $p=1,2$ are similar and, for simplicity, we focus on the case $p=2$. Making use of Lemma \ref{lem-HRep}, consider the test function
$$\Psi_s(g)= \mathbf 1_{V_s}(\tau) \prod_{i=1}^{k-1}\psi_s(z_i), \;\;g=z_1\dots z_{k-1}\tau \in \mathcal H_k$$
where the set $V_s\subset \mathbb S_0(Y)$ will be chosen later.
Obviously, we have
$$\|\Psi_s\|_2^2=|V_s| \|\psi_s\|_2^{2(k-1)}.$$
Next, we want to estimate $\mathcal E_{\mathcal H_k,q}(\Psi_s,\Psi_s)$ from above.  This involves computing the products
$h_{u,v}^m z_1\dots z_{k-1} \tau$  where $1\le u<v\le k$ and $m\in \mathbb Z$.  By inspecting the commutator relations  between 
$h^z_{i,k}$ and $h^m_{u,v}$, one finds that
$$h_{u,v}^m z_1\dots z_{k-1} \tau = z_1'\dots z_{k-1}' \tau'$$
where
$$ z'_i=z_i+\epsilon_{i,u,v} m  \mbox{ with }  \epsilon_{i,u,v}=\left\{\begin{array}{cl}0 &\mbox{ if } i\not\in \{u,v\}\\ 1 &\mbox{ if } i=u\\
-1& \mbox{ if } i=v \end{array}\right.$$
and
$$\mbox{support}(\tau') \subset \mbox{support}(\tau) \cup \mathbf S\left(2|m|+\sum_1^{k-1}|z_i|\right).$$
We choose 
$$V_s=\left\{\tau: \mbox{support}(\tau)\subset \mathbf S((k+3)r(s))\right\}$$
so that 
$$|z_i|\le r(s),\; |m|\le 2r(s) \mbox{ and }\tau\in V_s  \mbox{ implies } \tau' \in V_s.$$
Write
\begin{eqnarray*}\lefteqn{2\mathcal E_{\mathcal H_k,q}(\Psi_s,\Psi_s)=  \sum_{\Psi_s(g)\Psi_s(hg)\neq 0}|\Psi_s(hg)-\Psi_s(g)|^2 q(h) }\hspace{1.5in} &&\\
&& +
2\sum_{\Psi_s(g)\neq 0,\Psi_s(hg)=0}|\Psi_s(g)|^2q(h).
\end{eqnarray*}
By inspection, the right most term on the right-hand side is bounded above by
$$4|V_s|  {k\choose 2} ^{-1} \|\psi_s\|_2^{2(k-2)}\sum_{1\le u<v\le k}\; \sum_{\psi_s(z)\neq 0,\psi_s(z+y)=0}|\psi_s(z)|^2p_{u,v}(y).$$
Consider $g,h$ such that $\Psi_s(g)\Psi(hg)\neq 0$ with  $q(h)\neq 0$  and write $g=z_1\dots z_{k-1}\tau$, $h=h_{u,v}^m$.  We must have
$$|z_i|\le r(s),  \;\;|m|\le 2r(s) \mbox{ and }\;\mbox{support}(\tau)\subset \mathbf S((k+3)r(s)).$$  Let us consider the more difficult case when $v\neq k$.
Then, we have
\begin{eqnarray*}\lefteqn{|\Psi_s(hg)-\Psi_s(g)|^2=}&&\\ && \mathbf 1_{V_s}(\tau)\prod_{j\not\in \{u,v\}}|\psi_s(z_j)|^2
\left| \psi_s(z_u+m)\psi_s(z_v-m)-\psi_s(z_u)\psi_s(z_v)\right|^2
\end{eqnarray*}
and
\begin{eqnarray*}\lefteqn{\left| \psi_s(z_u+m)\psi_s(z_v-m)-\psi_s(z_u)\psi_s(z_v)\right|^2 \le}&&\\
&& 2\left(|\psi_s(z_u+m)|^2|\psi_s(z_v-m)-\psi_s(z_v)|^2  +
 |\psi_s(z_v|^2|\psi_s(z_u+m)-\psi_s(z_u)|^2     \right).
\end{eqnarray*}
Using this inequality and summing up we obtain that
$$ \sum_{\Psi_s(g)\Psi_s(hg)\neq 0}|\Psi_s(hg)-\Psi_s(g)|^2 q(h) $$
is bounded by
$$2 {k\choose 2} ^{-1} |V_s|\|\psi_s\|_2^{2(k-2)} \sum_{1\le u<v\le k}\;\sum_{\psi_s(z)\psi_s(z+y)\neq 0}|\psi_s(z+y)-\psi_s(z)|^2p_{u,v}(y).
$$
Putting the different terms together yields
$$2\mathcal E_{\mathcal H_k,q}(\Psi_s,\Psi_s)\le 4|V_s| {k\choose 2} ^{-1} \|\psi_s\|_2^{2(k-2)}\sum_{1\le u<v\le k} \mathcal E_{\mathbb Z,p_{u.v}}(\psi_s,\psi_s)$$
and
$$\frac{\mathcal E_{\mathcal H_k,q}(\Psi_s,\Psi_s)}{\|\Psi_s\|_2^2}\le  2s.$$
\end{proof}

The following results describe the spectral profile of $q$ in the special case when each $p_{i,j}$ is one of the measures
$$\xi_\alpha,\;
 \alpha\in (0,\infty)\cup \{\mathfrak s,\mathfrak t\} $$  where
 $$\xi_\alpha(m)=\left\{\begin{array}{cl}c_\alpha (1+|m|)^{-\alpha-1}  & \mbox{ if } \alpha\in (0,\infty)\\
 3^{-1}\mathbf 1_{\{-1,0,1\}}(m) & \mbox{ if } \alpha=\mathfrak s\\
 \delta_0 & \mbox{ if } \alpha =\mathfrak t.\end{array}\right.$$
 Set 
 $$ \rho_\alpha(s)= \left\{\begin{array}{cl} 
  s^{-1/\alpha} & \mbox{ if }\alpha\in (0,2)\\
 s^{-1/2} [1+\log (1/s)]^{1/2}&\mbox{ if } \alpha=2\\
 s^{-1/2} & \mbox{ if } \alpha\in (2,\infty)\cup \{\mathfrak s\} \\
 0 &\mbox{ if } \alpha=\mathfrak t  \end{array}\right.$$
 
\begin{theorem} \label{th-Houghtonstable}
On the Houghton group $\mathcal H_k$, let the probability measure $q_\mathfrak p$ at {\em (\ref{Houghtonq})} be such that 
for each $1\le i<j\le k$,
$p_{i,j}=\xi_{\alpha_{i,j}}$, $\alpha_{i,j}\in (0,\infty)\cup \{\mathfrak s,\mathfrak t\}$.  Assume that at least two $\alpha_{i,j}$ are different from $\mathfrak t$
and set 
$$\rho(s)= \max\{\rho_{\alpha_{i,j}}(s): 1\le i<j\le k\}.$$ 
Then , for all $v,s>0$, we have
$$\Lambda_{2,\mathcal H_k,q}(v) \simeq  \rho^{-1}\left(\frac{\log(1+ v)}{\log(1+\log(1+ v))}\right)  .$$
In particular, 
$$\Lambda_{2,\mathcal H_k}(v) \simeq  \left(\frac{\log(1+v)}{\log(1+\log(1+ v))}\right)^2.$$\end{theorem}
\begin{proof} The lower estimate for the spectral profile is obtained by Theorem  \ref{th-Houghton-1}. The matching upper bound follows from 
Theorem \ref{th-Houghton-2}. All we need to check is that  $\Lambda_{2,\mathbb Z,\xi_\alpha}(v)\simeq \rho^{-1}_\alpha (v)$  (in order to apply Theorem \ref{th-Houghton-1}) and that the functions $f_t(z)= (t-|z|)_+$ provide good test functions  in the sense that
$$\frac{\mathcal E_{\xi_\alpha}(f_t,f_t)}{\|f_t\|_2^2} \simeq  \rho^{-1}_\alpha(t).$$
See \cite[A.2]{SCZ-AOP2016}.
\end{proof}
\begin{remark} One can prove a version of Theorem \ref{th-Houghtonstable} dealing with  the isoperimetric profile instead of the spectral profile by using a similar line of reasoning and the results of \cite[A.2]{SCZ-AOP2016}. 
\end{remark}

\section{Other examples: Schreier graphs and star extensions}

Star extension of Cayley graphs (Example \ref{star}) and 
pocket and rooted extensions based on Schreier graphs are, in general, more difficult to handle than the  pocket and rooted extensions of Cayley graphs
treated in the previous section.  In this section we look, successively, at rooted extensions based on Schreier graphs and 
at star extensions of Cayley graphs.

Structurally, what makes a rooted extension $\Gamma$ of a Cayley graph easier to handle is the fact that the permutations of the underlying set $X$ appearing in $\Gamma$  can be reduced to translations along the constituent subgroups associated with the original Cayley graphs times finite support permutations of $X$. In the general Schreier graph case, even so any element of $\Gamma$ appears to ``look like'' a translation at infinity in each of the constituent Schreier graphs, it is not possible to assign uniquely an actual element of the corresponding subgroup of $\Gamma$ to capture this effect. Nevertheless, in some simple cases when the main feature governing the behavior of random walks on $\Gamma$ is the volume growth functions  of the constituent Schreier graphs, it is possible to obtain satisfactory results via a rather coarse approach explained in the next section.  This same approach applies as well to the study of star extensions of Cayley graphs.

Note that the results obtained for  rooted extensions based on Cayley graphs allow us a large variety of measure  $q$ including the possibility of measures 
with infinite support.  The results obtained in this section are restricted to finitely supported measures (up to comparisons of forms).

\subsection{Comparison with random three cycles} \label{sec-3cycles}

Consider $\ell$ labelled rooted connected graphs $(X_i,E_i,m_i,o_i)$, $i=1,2,\dots,\ell$ (with distinct labellings). Let $(X,E,m)$ be the labelled graph  with vertex set $X= (\cup_1^\ell X_i\setminus\{o_i\}  )  \cup \{o\}$ corresponding to identifying the points $o_1,o_2,\dots,o_\ell$. Let $\Gamma$ be the corresponding subgroup of $\mathbb S(X)$. The group $\Gamma$ contains a copy of each $\Gamma_i$
where $\Gamma_i$ is the group defined by $(X_i,E_i,m_i)$ and also a copy of  $\mathbb A_0(X)$ with
$$\mathbb A_0(X)
= \langle[ \Gamma_i,\Gamma_j]^\Gamma; i\neq j\rangle.$$
Indeed, one verifies by inspection that for any two elements $g_i,g_j\in \Gamma$, $g_i\in \Gamma_i,g_j\in \Gamma_j$, which move $o$, $i\neq j$,
we have
$$ [g_i,g_j]=g_ig_jg_i^{-1} g_i^{-1}=(o,g_i\cdot o,g_j\cdot o),$$ as in Lemma \ref{lem-com},
$$g[g_i,g_j]g^{-1}=\left\{\begin{array}{ll}(g\cdot o,g_i\cdot o,g_j\cdot o) &\mbox{ if } g\in \Gamma_k, k\not\in \{i,j\}\\
(g\cdot o, g g_i\cdot o,g_j\cdot o) &\mbox{ if } g\in \Gamma_i\\
(g\cdot o,g_i\cdot o, gg_j\cdot o) &\mbox{ if } g\in \Gamma_j,\end{array}\right.$$
and (still assuming $g_j\cdot o\neq o$)
$$g_jg [g^{-1}g',g_j]g^{-1}g_j^{-1}=(o,g\cdot o,g'\cdot o), \;\;g,g'\in \Gamma_i,\; g_j\in \Gamma_j,\; i\neq j,\; g^{-1}g'\cdot o\neq o.$$

Any element $\gamma$ of $\Gamma$ can be written in the form  $\gamma=g_1\dots,g_\ell \tau$ with $g_i\in \Gamma_i$ and $\tau \in \mathbb A_0(X)$ but this can possibly be done in many different ways since pairs of elements in a given $\Gamma_i$ may only differ via a permutation of finite support of $X_i$.
Note that $\gamma=g_1\dots,g_\ell \tau$ with $g_i\in \Gamma_i$ and $\tau \in \mathbb A_0(X)$ belongs to $\mathbb S_0(X)$ if and only if each $g_i$ reduces  to  a finite permutation on $X_i$. This shows that $\Gamma$ contains the full symmetric group with finite support $\mathbb S_0(X)$ exactly when at least one of the groups $\Gamma_i$ contains an odd  permutation with finite support.  The following Proposition is tailored to cover the situations described above but is framed in a much more general setting.

\begin{proposition}\label{pro-cycles}
Let $\Gamma$ be a finitely generated group with finite generating set $T=\{\theta_1^{\pm 1},\dots,\theta_k^{\pm 1}\}$. Assume that $(X,E,m,o)$
is a connected labelled rooted Schreier graph for $(\Gamma,T)$. Let $d$ be the graph distance between two points of $X$ and set $B_X(o,r)=\{x\in X: d(o,x)\le r\}$. Assume that  either 
$\mathbb S_0(X)\subset \Gamma$ with
$$|(o,x)|_T \le D \max\{d(o,x)\}$$
or, more generally, that 
$\mathbb A_0(X)\subset \Gamma$ with
$$|(x,y,z)|_T\le D \max\{d(o,x),d(o,y),d(o,z)\}.$$
Then there is constants $c_1$ such that 
$$\Lambda_{1,\Gamma,\mathbf u}(v)\ge  c_1 r^{-1}  \mbox{ for all } v<  \sqrt{|B_X(o,r)|!} $$
where $\mathbf u=|S|^{-1}\mathbf 1_S$ is the uniform measure on the symmetric generating set $S$.
\end{proposition}
\begin{remark}
An acceptable lower bound on the spectral profile is obtained by applying the general inequality
$\Lambda_{2,\Gamma,\mathbf u}(v)\ge\frac{1}{2} \Lambda_{1,\Gamma,\mathbf u}(v)^2$. 
\end{remark}
Note that the first case is actually covered by the second case. In the first case where one can use the transpositions $(o,x)$, a simpler direct proof using comparison with the ``transpose $o$ and $x$'' random walk can be given. It follows the same line of reasoning described below for the second case. 

Let $N=|B_X(o,r)|$. Let $\mu_N$ be the uniform measure on all three-cycles $(x,y,z)$ with $x,y,z\in B_X(o,r)$.  Regarding $\mu_N$ as a measure on $\mathbb A_N$ (the finite alternating group on $N$ objects), we know that $\mu_N^{(t)}$ converges to $2/(N!)$ as $t$ tends to infinity.
It is well known that this walk can be analyzed in details in a way similar to what was done for the random transposition walk in \cite{DS}.  In particular, it is proved in \cite{Roussel} (see also \cite{Roichman,HSZ}) that there exists a constant $C$ such that, for all $t\ge \frac{2}{3}N(\log N +c)$, $c>0$, 
$$ (N!/2) |\mu_N^{(t)}(\id)- 1|\le  C e^{- 2c} .$$
In particular,
$$\mu_N^{(t_N)}(\id)\le (1+C)\frac{2}{N!},\;\;t_N= \frac{2}{3}N\log N.$$
\begin{lemma} \label{lem-AN}
Fix $\epsilon\in (0,1)$. There exists $N_0$ such that for $v\in (0,(N!)^\epsilon)$ and $N\ge N_0$, we have
$$\Lambda_{1,\mathbb A_N,\mu_N}(v)\ge \Lambda_{2,\mathbb A_N,\mu_N}(v)\ge (1-\epsilon).$$
\end{lemma}
\begin{proof} Using the trace formula for the random walk with Dirichlet boundary condition on a set $U\subset \mathbb A_N$, we have 
$$\mu_N^{(t)}(\id) \ge \frac{1}{|U|}\exp(-\lambda(\mu_N,U) t)$$
where 
$$\lambda(\mu_N,U)= \inf\left\{\mathcal E_{\mu_N}(\phi,\phi):  \mbox{support}(\phi) \subset U, \|\phi\|_2=1 \right\}$$
is the lowest eigenvalue of $\delta_{\id}-\mu_N$ in $U$ with Dirichlet boundary condition.  In particular,
$$\lambda(\mu_N,U) \ge t_N^{-1}(\log (N!) - \log |U|  -\log (2(1+C))).$$
Since 
$$\Lambda_{1,\mathbb A_N, \mu_N}(v) \ge \Lambda_{2,\mathbb A_N,\mu_N}(v)=\inf\{\lambda(\mu_N,U): |U|\le v\}$$
the desired result follows. 
\end{proof}
\begin{proof}[Proof of Proposition {\ref{pro-cycles}}]  Consider the two probability measures $\mathbf u$ and $\mu_N$ on the group $\Gamma$.
The hypothesis $|(x,y,z)|_S\le D \max\{d(o,x),d(o,y),d(o,z)\}$  for $x,y,z\in B_X(o,r)$ and a simple comparison technique  imply that
$$\frac{1}{2}\sum_{h,g\in \Gamma}|f(gh)-f(h)|\mu_N(g)\le D r \, \frac{1}{2}\sum_{h,g\in \Gamma}|f(gh)-f(h)|\mathbf u(g).$$
Hence the conclusion of Proposition \ref{pro-cycles} follows readily from the result stated in Lemma \ref{lem-AN} (here we choose $\epsilon=1/2$
in Lemma \ref{lem-AN}).
\end{proof}

\subsection{Example: pocket extensions based on Schreier graphs}
Proposition \ref{pro-cycles} applies easily to the pocket extension based on a Schreier graph.   Please note that a given group $G$ may be defined by any one of its actions on a variety of different Schreier graphs. The finitely generated group $\Gamma$ defined by the rooted labelled Schreier graph $(X^*,E^*,m^*)$ as in Example \ref{exa-P} obtained form a given rooted labelled Schreier graph $(X,E,m)$ that defines $G$ is an object that depends not only on $G$ but on $(X,E,m)$. 

Let $(X,E,m)$ be a rooted labelled Schreier  graph defining a group $G$.
Let $\Gamma$  be the finitely generated group defined by the rooted labelled Schreier graph $(X^*,E^*,m^*)$  (the pocket extension of $(X,E,m)$) 
as defined in Example \ref{exa-P}. 

\begin{corollary}[of Proposition \ref{pro-cycles}]  \label{cor-*S}
Let $V_\bullet$ be the volume growth function of $(X,E,m)$ at the root. Set $V_\bullet^{-1}(t)=\inf\{s: V_\bullet(s)\ge t\}$. Let $\Gamma$ be as above. We have
$$\Lambda_{1,\Gamma}(v) \gtrsim \frac{1}{V_\bullet^{-1}\left( \frac{ \log(1+ v)}{\log(1+\log(1+ v))}\right)},\;\;
\Lambda_{2,\Gamma}(v) \gtrsim \frac{1}{\left[V_\bullet^{-1}\left( \frac{ \log(1+ v)}{\log(1+\log(1+ v))}\right)\right]^2}.$$ 
\end{corollary}

When specializing to the case when $(X,E,m)$ is a Cayley graph of $G$ (in which case $\Gamma=G_{\circledast}$ is the pocket extension of $G$),
this result is weaker than the result provided by Theorem \ref{th-beta23} (and Theorem \ref{th-pocket}.  It is sharp only when the isoperimetric profile $\Lambda_{1,G}$ of $G$ 
satisfies $\Lambda_{1,G}(v)\simeq 1/V^{-1}(v)$ where $V$ is the volume growth function of $G$.  On the other hand, the above result apply in much greater generality.

\begin{remark} It is straightforward to generalize Corollary \ref{cor-*S} to the rooted gluing of $\ell$ labelled Schreier graphs. The statement is 
the same with $V_\bullet= \max_{1\le i\le \ell}\{V_i\}$ where each $V_i$ is the rooted volume function on $(X_i,E_i,m_i)$. 
\end{remark}

\subsection{Example: star extensions of a Cayley graph}

Proposition applies nicely to the star extension $\Gamma=(G,S)_\star$ of a labelled Cayley graph $(G,S)$ (see Example \ref{star}).  Indeed, in this case $\Gamma=G\ltimes \mathbb S(G)$  and  if $x= \sigma_1\dots \sigma_\ell$ in $G$, $\sigma_i\in S\cup S^{-1}$ 
then we can write the transposition $(e,x)$ in the form
\begin{equation}\label{eq-star}
(e, x) = \sigma_1 t_{\sigma_1} \dots \sigma_{\ell-2} t_{\sigma_{\ell-2}}\sigma_{\ell-1}  t_{\sigma_\ell} \sigma^{-1}_{\ell-1} t_{\sigma_{\ell-1}}\dots \sigma_2^{-1}t_{\sigma_2}\sigma_1^{-1}t_{\sigma_1}   \end{equation}
where $$t_\sigma= \left\{\begin{array}{cl}  t_i =(e,s_i)& \mbox{ if } \sigma=s_i \\s_i^{-1}t_is_i=(e,s_i^{-1})& \mbox{ if } \sigma=s_i^{-1}
\end{array}\right.$$
This shows that  the transposition $(e,x)$ as length at most $8 \ell$. In other words, if $x$ has length $|x|_S$ in $(G,S)$ then $(e,x)$ has length at most 
$8|x|_S$ in $(\Gamma,T)$ where $T=S\cup \{t_i=(e,s_i): 1\le i\le k\}$.  
\begin{theorem}  \label{th-starlow}
Let $(G,S)$ be a labelled Cayley graph with volume growth function $V$ and $V^{-1}(t)=\inf\{s: V(s)\ge t\}$. Let $\Gamma=(G,S)_\star$ be its the star extension. Then we have
$$\Lambda_{1,\Gamma}(v)\gtrsim \frac{1}{V^{-1}\left(\frac{\log(1+ v) }{\log(1+\log(1+ v))}\right)},\;\;
\Lambda_{2,\Gamma}(v)\gtrsim \frac{1}{\left[V^{-1}\left(\frac{\log(1+ v) }{\log(1+\log(1+ v))}\right)\right]^2}.$$
\end{theorem}
 The next result provides an upper-bound.
 \begin{theorem} \label{th-starup}
  Let $(G,S)$ be a labelled Cayley graph. Let $\Gamma=(G,S)_\star$ be its  star extension.  Then we have
 $$\Lambda_{p,\Gamma}(v) \lesssim    \Lambda_{p,G}\left(  \frac{\log(1+ v)}{\log(1+\log(1+ v))} \right).$$
 \end{theorem}
 \begin{corollary} [of Theorems \ref{th-starlow}-\ref{th-starup}] The star extension $\Gamma=(G,S)_\star$ of any Cayley graph $(G,S)$ of a polycyclic group $G$ satisfies
 $$\Lambda_{p,\Gamma}(v) \simeq    \Lambda_{p,G}\left(  \frac{\log (1+v)}{\log(1+\log(1+ v))} \right),\;p=1,2.$$
  \end{corollary}
\begin{proof}[Proof of Theorem \ref{th-starup}] To estimate $\Lambda_{p,\Gamma}$, we pick the finitely supported measure $\nu=\frac{1}{2}(\nu_1+\nu_2)$ where $\nu_1$ is the uniform
measure on $S$, the generating set of $G$ viewed as a subgroup of $\Gamma$, and $\nu_2$ is the uniform measure on the $k$ transpositions $t_i=(e,s_i)$
also viewed as elements in $\Gamma$. 
Since $\Gamma=G\times \mathbb S(G)$, we can try to use a test function of the form
$$\psi(\gamma)= \mathbf 1_{V}(\tau)\phi(g),\;\;\gamma=(g,\tau).$$ 
We pick $\phi$ to be a good test function for $\Lambda_{p,G}(v)$ so that  $U=\mbox{support}(\phi)$ in $G$ has size at most $v$  and
$$\frac{1}{2|S|} \sum_{g\in G,s\in S}|\phi(sg)-\phi(g)|^2 \le  \eta \sum_{g\in G}|\phi(g)|^2,\;\; \eta=2\Lambda_{p,G}(v).$$
We then pick 
$$V= \bigcup_{s\in S} U^{-1}s.$$   We give the details in the case $p=2$ (the case $p=1$ is very similar). Write
\begin{eqnarray*}\mathcal E_\nu(\psi,\psi)&=&\frac{1}{2}\sum_{\gamma,z\in \Gamma}|\psi(z\gamma)-\psi(\gamma)|^2\nu(z)\\
&=& \frac{1}{4}\sum_{\gamma,z\in \Gamma}|\psi(z\gamma)-\psi(\gamma)|^2\nu_1(z) +\frac{1}{4}\sum_{\gamma,z\in \Gamma}|\psi(z\gamma)-\psi(\gamma)|^2\nu_2(z) .\end{eqnarray*}
The first term in this sum is obviously equals to
$$  \frac{|V|}{4|S|}\sum_{g\in G,s\in S}|\phi(sg)-\phi(g)|^2 $$
which bounded below by $ \frac{\eta |V|}{2}\|\phi\|_2^2= (\eta/2) \|\psi\|_2^2$.
For the remaining term, write
\begin{eqnarray*}
\frac{1}{4}\sum_{\gamma,z\in \Gamma}|\psi(z\gamma)-\psi(\gamma)|^2\nu_2(z) &=&
\frac{1}{4k}\sum_{g;\tau; 1\le i\le k}|\mathbf 1_V(g^{-1}t_i g \tau)-\mathbf 1_V(\tau)|^2|\phi(g)|^2\end{eqnarray*}
Note that  for $ g^{-1}t_i g$ is the transposition $ (g^{-1},g^{-1}s_i)$. It follows that when $g\in U=\mbox{support}(\phi)$,  $g^{-1}t_ig\tau$ and $\tau$  are supported  either  both in $V$  or both in $G\setminus V$. This means that   $\frac{1}{4}\sum_{\gamma,z\in \Gamma}|\psi(z\gamma)-\psi(\gamma)|^2\nu_2(z) =0$.   Thus we have found a function $\psi$ such that
$$\mathcal E_\nu(\psi,\psi) \le \frac{\eta}{2} \|\psi\|_2^2$$
and which has a support of size at most    $ (kv)!$. This yields the desired result.
\end{proof}

\begin{remark}The upper bound on $\Lambda_{p,\Gamma}$, $\Gamma=(G,S)_\star$,
can also be obtained indirectly by noting that $\Gamma$ is a subgroup of $G_{\circledast}$ and using Theorem \ref{th-pocket}. It is to be noted that Theorem
\ref{th-pocket} allows for starting with an arbitrary symmetric probability measure on $G$ but that the results concerning the star extension $(G,S)_\star$ are obtained only  for
finitely supported  symmetric probability measures on $G$.
\end{remark}

\section{Pocket extension of Schreier graphs: the case of bubble groups}

\subsection{Bubble groups } \setcounter{equation}{0}
\label{sec:Explicit-estimates}\label{sec-bubble}
The family of the so-called bubble groups was considered in \cite{Kotowski2014}.  See also \cite{AmirKozma,SCZ-Rev}.
   
\begin{figure}[h]
\begin{center}\caption{A piece of the labelled Schreier graph of an infinite 
bubble group with $\mathbf a= (a_1,a_2,\dots), \mathbf b=(3,3,3\dots)$. }\label{B2}

\begin{picture}(340,100)(0,0) 

\put(25,50){\oval(170,15)[r]}
\put(120,50){\circle{20}}
\put(110,50){\circle*{3}}
\multiput(127,56.5)(0,-13){2}{\circle*{3}}
\multiput(105,57.5)(-20,0){5}{\circle*{3}}
\multiput(105,42.5)(-20,0){5}{\circle*{3}}
\multiput(105,60.5)(-20,0){5}{\circle{6}}
\multiput(105,39.5)(-20,0){5}{\circle{6}}

\multiput(30,57.5)(-10,0){3}{\line(-1,0){6}}
\multiput(30,42.5)(-10,0){3}{\line(-1,0){6}}

\put(225,64){\oval(200,15)}
\multiput(127,71.5)(19.4,0){11}{\circle*{3}}
\multiput(127,74.5)(19.4,0){10}{\circle{6}}
\multiput(146.4,56.5)(19.4,0){10}{\circle*{3}}
\multiput(146.4,53.5)(19.4,0){10}{\circle{6}}
\put(329.3,79){\oval(20,20)[bl]} 

\put(333,82){\oval(20,20)[bl]}
\put(323,82){\vector(0,1){2}}
\put(328,82){\makebox(0,0){$b$}}

\put(225,36){\oval(200,15)}
\multiput(146,42.5)(19.4,0){10}{\circle*{3}}
\multiput(146,45.5)(19.4,0){10}{\circle{6}}
\multiput(127.4,28)(19.4,0){11}{\circle*{3}}
\multiput(127.4,25)(19.4,0){10}{\circle{6}}
\put(330,21){\oval(18,18)[tl]}

\put(333,18){\oval(18,18)[tl]}
\put(333,27){\vector(1,0){2}}
\put(328,20){\makebox(0,0){$b$}}

\put(95,63){\makebox(0,0){$a$}}
\put(90,59.5){\vector(1,0){10}}

\put(156,22){\makebox(0,0){$a$}}
\put(161,26){\vector(-1,0){10}}

\put(111.5,58.5){\put(3,6.5){\makebox(0,0){$b$}}
\put(0,0){\vector(2,1){10}}}

\put(154,77){\makebox(0,0){$a$}}
\put(152,74){\vector(1,0){9}}
\put(186,82){\makebox(0,0){$b$}}

\end{picture}\end{center}\end{figure}
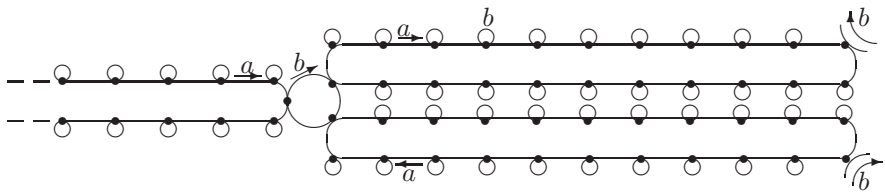

Let $\mathbf{a}=(a_{1},a_{2},...)$ and $\mathbf{b}=\left(b_{1},b_{2},..\right)$ 
be two natural integer infinite sequences. 
The ``bubble group'' 
$\Gamma_{\mathbf a,\mathbf b}$ is associated with the tree like bubble graph 
$X_{\mathbf a,\mathbf b}$ 
were $X_{\mathbf a,\mathbf b}$ is obtained from the 
rooted tree $\mathbf T_{\mathbf b}$ 
with forward degree sequence $(1,b_1-1,b_2-1,\dots)$  as follows. 
Each edge at level $k\ge 1$ in the tree (we make the convention that
the level of an edge is the level of the child on that edge)
is replaced by a cycle of length $2a_k$ called a bubble. 
Each vertex at level $k\ge 1$ (we ignore the root which is now part of a circle 
of length $2a_1$) is blown-up 
to a $b_k$-cycle with each vertex of this cycle inheriting one of the 
associated $2a_{k+1}$-cycle. These $b_k$-cycles are called branching cycles. 
Finally, 
at each vertex which belong only to a bubble (but not to a branching cycle), 
we add a self loop.  The vertex set of the graph $X_{\mathbf a,\mathbf b}$ 
can be parametrized  using pairs $(w,u)$ with $w$ a finite word in 
$$\{\emptyset\} \cup (
\cup _{k=1}^\infty \{1,\dots,b_1-1\}\times\{1,\dots,b_2-1\}\times \cdots \times \{1,,\dots,b_k-1\})$$
and $u\in \{0,\dots,2a_{k+1}-1\}$ if $|w|=k$.  By definition, the vertex 
$\mathfrak o=\emptyset$ is the root. 

In the graph  $X_{\mathbf a,\mathbf b}$, we call ``level k'' 
the set of all the vertices $(w,u)$ with $|w|=k-1$, $0\le u\le 2a_{k}-1$.
If all the $a_k$ are distinct, this is the set of all vertices that 
belong to a bubble of length $2a_k$.  We say that a branching cycle
is at ``level k'' if it is attached  at the far end (i.e., 
furthest away from $o$) of a level-$k$ bubble.
Note that the vertices of any branching cycle at level $k$ are 
parametrized as follows:  \begin{itemize}
\item
$(w', a_k)$ with $|w'|=k-1$ for the vertex closest to the root $o$, a vertex 
which also belongs  to a level-$k$ bubble,
\item
$(w'z,0)$ with $z\in \{1,\dots, b_k-1\}$ for the other vertices on 
that branching cycle, each of which also belongs to a level-$(k+1)$ bubble.
\end{itemize}
  We let 
$$\mathfrak b(w')=\{(w',a_k),(w'1,0),\dots,(w'(b_k-1),0)\}$$  
denote the branching cycle at $(w',a_k)$. 

Having chosen an orientation along each cycle (say, clockwise), we label each 
edge of the bubble with the letter $a$ and each edge of the branching 
cycle with the letter $b$.

The group $\Gamma_{\mathbf a,\mathbf b}$ is a subgroup of 
the (full) permutation group of the vertex set of $X_{\mathbf a,\mathbf b}$ 
generated 
by two elements $\alpha$ and $\beta$.  Informally, $\alpha$ rotates the 
bubbles whereas $\beta$ rotates the branching cycles.  Formally, the action of 
the permutation $\alpha$ (resp. $\beta$) on any vertex $x$ in 
$X_{\mathbf a,\mathbf b}$ 
is indicated by the oriented labeled edge at $x$ labelled with an $a$ 
(resp. a $b$). Obviously, we can replace the edge labels $a, b$ with the 
group elements $\alpha,\beta$, once these are defined.   

These groups are somewhat mysterious.  We know they have exponential volume growth when all $b_i$ are at least $3$, that the groups $\Gamma_{\mathbf a,\mathbf b}$ are non-amenable when both sequence $\mathbf a,\mathbf b$ are bounded and that they are amenable when $\liminf a_i=\infty$.
Typically, their isoperimetric and spectral profiles are not precisely known (modulo the usual equivalence relation $\simeq$).  A more detailed description is given in \cite{SCZ-Rev} and also in \cite{AmirKozma}, especially the appendix of this paper written by Nicolas Matte Bon. 
Here we will focus in the case when the sequence $\mathbf b=(3,3,3,\dots)$ and the sequence $a_i$ is strictly increasing. The main case 
of interest for us is when $a_i\approx 2^{\kappa i}$ for some fixed parameter $\kappa\in (0,\infty)$.  
Because we only consider the case $\mathbf b=(3,3,3,\dots)$,
we will use the simplified notation $X_{\mathbf a},\Gamma_{\mathbf a}$.  We equip $\Gamma_\mathbf a$ with the symmetric probability measure 
$\mathbf u$ which is the
uniform measure on $\{\alpha^{\pm1},\beta^{\pm 1}\}$ and with the associated Dirichlet form $\mathcal E_{\mathbf u}$ (multiplication on left).

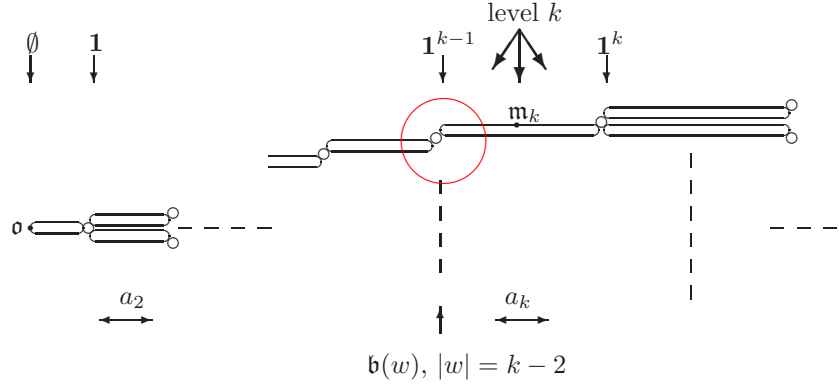
\begin{figure}[h]
\begin{center}\caption{Sketch 
of the Schreier graph $X$: levels, $\mathfrak b(w)$, $\mathfrak m_k$.
Details of the red circle region are shown in Figure \ref{BW2}.}\label{BW0}

\begin{picture}(320,150)(0,0) 

\put(10,0){\put(-5,50){\makebox(0,0){$\mathfrak o$}}
\put(0,50){\circle*{2}}
\put(10,50){\oval(20,4.2)}
\put(22,50){\circle{4}}

\put(37.5,53){\oval(30,4.2)}
\put(37.5,47){\oval(30,4.2)}
\put(54,55.5){\circle{4}}
\put(54,44.2){\circle{4}}

\multiput(56,50)(10,0){4}{\line(1,0){5}} 

\put(24,115){\vector(0,-1){10} 
\makebox(-5,10){$\mathbf 1$}}

\put(0,115){\vector(0,-1){10} 
\makebox(-5,10){$\emptyset$}}

\put(30,30){\makebox(17,-15){$a_2$} 
\put(-11,-15){\vector(1,0){10}} \put(-11,-15){\vector(-1,0){10}} 
}
}

\put(100,75){\put(0,0){\oval(40,4.2)[r]}
\put(21,3){\circle{4}}
}

\put(142,81){\put(0,0){\oval(40,4.2)}
\put(21.5,3.05){\circle{4}}
}

\multiput(165,68)(0,-10){4}{\line(0,-1){5}} 
\put(160,-5){\makebox(30,5){$\mathfrak b(w)$, $|w|=k-2$}} 
\put(165,10){\vector(0,1){10} 
}

\put(195,87){\put(0,0){\oval(60,4.2)}
\put(31,3){\circle{4}}}

\put(262,92.5){\put(0,1){\oval(70,4.2)}
\put(36,4){\circle{4}}
\put(0,-5.5){\oval(70,4.2)}
\put(36,-8.5){\circle{4}}

}

\multiput(260,78)(0,-10){6}{\line(0,-1){5}} 

\multiput(290,50)(10,0){3}{\line(1,0){5}}

\put(166,115){\vector(0,-1){10} 
\makebox(-2,10){$\mathbf 1^{k-1}$}}

\put(173,30){\makebox(43,-15){$a_k$} 
\put(-20,-15){\vector(1,0){10}} \put(-20,-15){\vector(-1,0){10}} 
}
\put(228,115){\vector(0,-1){10} 
\makebox(-4,10){$\mathbf 1^{k}$}}

\put(195,91){\makebox(4,4){$\mathfrak m_k$}}
\put(194,89){\circle*{2}}

\put(195,130){\makebox(4,4){level $k$}}
{\thicklines \put(195,125){\vector(2,-3){10}}
\put(195,125){\vector(0,-1){20}}
\put(195,125){\vector(-2,-3){10}}
}

{\color{red}  \put(163,83){\circle{30}}}

\end{picture}\end{center}\end{figure}

\begin{figure}[h]
\begin{center}\caption{Sketch 
showing 
$\color{red}\mathfrak N(\mathbf 1^{k-2},r)$,
$\mathfrak m_k$}\label{BW2}

\begin{picture}(320,130)(0,0)

\put(150,70){\circle{20}}
\multiput(140.5,66.5)(-10.6,0){10}{\circle*{3}}
\multiput(140.5,56.5)(-10.6,0){10}{\circle*{3}}
{\color{red}\multiput(140.5,66.5)(-10.6,0){3}{\circle{3.5}}
\multiput(140.5,56.5)(-10.6,0){3}{\circle{3.5}}}

\multiput(155.5,79.5)(10.6,0){12}{\circle*{3}}
\multiput(155.5,89.5)(10.6,0){12}{\circle*{3}}
{\color{red}\multiput(155.5,79.5)(10.6,0){3}{\circle{3.5}}
\multiput(155.5,89.5)(10.6,0){3}{\circle{3.5}}}

\multiput(157.5,62.5)(10.6,0){12}{\circle*{3}}
\multiput(157.5,52.5)(10.6,0){12}{\circle*{3}}
{\color{red} \multiput(157.5,62.5)(10.6,0){3}{\circle{3.5}}
\multiput(157.5,52.5)(10.6,0){3}{\circle{3.5}}}

\put(92,62){\oval(100,10)}
\put(216.5,58){\oval(120,10)}
\put(215,84){\oval(120,10)}

\put(96,81){\makebox(30,10){$(\mathbf 1^{k-2},a_{k-1})$}}
\put(127,80){\vector(1,-1){10}}

\put(108,96){\makebox(30,10){$(\mathbf 1^{k-1},0)$}}
\put(137,95){\vector(1,-1){13}}

\put(108,33){\makebox(30,10){$(\mathbf 1^{k-2}\mathbf 0,0)$}}
\put(141.5,46){\vector(1,1){12}}

\put(0,33){\makebox(30,10){$(\mathbf 1^{k-2},0)$}}
\put(33,44){\vector(1,1){10}}

\put(120.6,23.5){
\put(53,73){\makebox(30,10){$\mathfrak m_k$}}
\put(74,74){\vector(2,-1){10}}
{\color{blue}\put(87.6,66){\circle{5}}}}

\put(285,105){\makebox(30,10){$(\mathbf 1^{k-1},a_k)$}}
\put(285,104){\vector(-1,-1){10}}

\end{picture}\end{center}\end{figure}
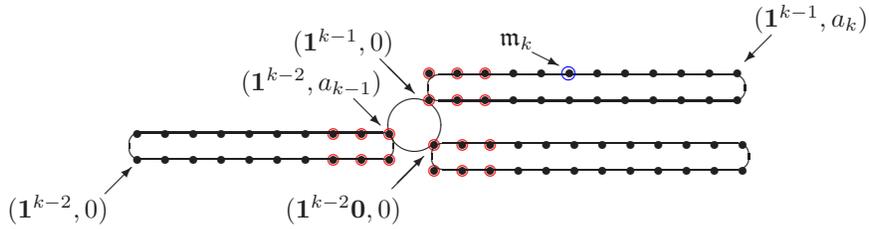

For simplicity of notation, we assume throughout that each entry $a_i$ of the sequence $\mathbf a$ is divisible by $4$.  Set
$$\mathfrak N (w,r)=\{x\in X: d(x,\mathfrak b(w)) \le r\},\;w\in 
\{1,2\}^{(\infty)},\;r>0 .$$
For any $k\le j$, $w$ of length $|w|=j$ 
and $0\le r \le a_{k-1}-1$, we have an obvious bijective map
$$\iota^w_k: \mathfrak N (w,r)\mapsto \mathfrak N(\mathbf 1^{k-1},r)$$ 
which can be used to identify these vertex sets. For a given level $k$, we set 
$$\mathfrak{m}_k=(\mathbf 1^{k-1},a_{k}/2),\;\;B_k(l)=\left\{ x\in X:\ d\left(x,\mathfrak{m}_k\right)\le 
l\right\}, \;0\le l\le (a_{k}/2)-1.$$

\subsection{Construction of test functions on $\Gamma_{\mathbf a}$ and the pocket extension of its Schreier graph}

\begin{definition}
For each $k\ge 8$ and $\ell \in (0, (a_k/2)-1)$, consider the set $U_k(\ell)$ of all elements $g\in \Gamma_{\mathbf a}$ such that
there exists a sequence $\gamma_1,\dots \gamma_q\in \{\alpha^{\pm 1},\beta^{\pm 1}\}$ such that $g=\gamma_q\dots \gamma_1 $
and, for all $1\le j\le q$,  $\gamma_j\dots\gamma_1\mathfrak m_k \in B_k( \ell)$. 
For any $s\in \{0,1,\dots, \ell_k\}$, with $\ell_k = (a_k/4)-1 $,  define $\psi_{k}$ on $\Gamma_{\mathbf a}$ by
$$\psi_k(g)=\left\{ \begin{array}{cl} 0  &\mbox{ if } g\not\in U_k(\ell_k)\\
(1- s/\ell_k)_+ & \mbox{ if } g\in U_k(\ell_k) \mbox{ and }  d(\mathfrak m_k,g\mathfrak m_k)=s.\end{array}\right.$$ 
\end{definition}

\begin{lemma} \label{lem-bubbletest}
The function $\psi_k$ satisfies
$$\|\psi_k\|_2^2= \frac{2\ell_k+ 2+ 1/\ell_k}{3(2\ell_k+1)}|U_k(\ell_k)| \mbox{ and }  \mathcal E (\psi_k,\psi_k) =  \frac{1}{2\ell_k(2\ell_k+1)}|U_k(\ell_k)|.$$
In particular
$$\frac{\mathcal E_{\mathbf u}(\psi_k,\psi_k)}{\|\psi_k\|_2^2}\le \frac{3}{2\ell_k^2}.$$
\end{lemma}
\begin{proof} The sets $U_{k}(\ell_k, t)=\{g\in U_{k}(\ell_k): g\mathfrak m_k =\alpha^{t}\mathfrak m_k\}$, $t\in \{0, \pm 1,\dots, \pm \ell_k\}$,
form a partition of $U_k(\ell_k)$ and they all have the same cardinality because one can check that $\alpha ^{-t} U_k(\ell_k,t)=U_k(\ell_k,0)$. 
It follows that 
\begin{eqnarray*}
\|\psi_k\|_2^2&=&|U_k(\ell_k,0)|\sum_{-\ell_k\le t\le \ell_k} (1-t/\ell_k)^2\\& = &\frac{1}{3}(2\ell_k+ 2+ 1/\ell_k)|U_k(\ell_k,0)
=\frac{2\ell_k+ 2+ 1/\ell_k}{3(2\ell_k+1)}|U_k(\ell_k)|,\end{eqnarray*}
and 
\begin{eqnarray*}
8\mathcal E_{\mathbf u}(\psi_k,\psi_k)&=& \sum_{g}(|\psi_k(\alpha g)-\psi_k(g)|^2 +\psi(\alpha^{-1} g)-\psi_k(g)|^2 )\\
&=& \sum _{-\ell_k\le t<\ell_k} (1/\ell_k)^2 |U_k(\ell_k,0)| +  \sum _{-\ell_k< t\le \ell_k} (1/\ell_k)^2|U_k(\ell_k,0)|\\
&=& \frac{4}{\ell_k}|U_k(\ell_k,0)|= \frac{4}{\ell_k(2\ell_k+1) }|U_k(\ell_k)|. \end{eqnarray*}  
\end{proof}

In order to use Lemma \ref{lem-bubbletest} to estimate the spectral profile, we need to estimate from above the size of the support of $\psi_k$.
Set $s_{k}=\sum_1^{k} a_j$

\begin{lemma} \label{lem-testsupport}
For any $k\ge 8$, any element $g \in U_k(\ell_k,0)\subset \Gamma_{\mathbf a}$, viewed as an element of $\mathbb S(X_{\mathbf a})$,
is the product of permutations  supported in the disjoint  finite sets $B(\mathfrak o, s_{k-1}+\ell_k)$  and $\mathfrak N_j(w,\ell_k)$, $|w|\ge k $. Moreover,
The factor supported in  $\mathfrak N(\mathbf 1^k,\ell_k)$ determines uniquely all the factors supported one each $\mathfrak N(\mathbf 1^k,\ell_k)$, $|w|\ge k$, via the isomorphisms  $\iota_{k+1}^w: \mathfrak N(w,\ell_k)\ra \mathfrak N(\mathbf 1^k,\ell_k)$.  In particular,
$$|U_k(\ell_k,0)|\le (|B(\mathfrak o, s_{k-1}+\ell_k)|!) \times (|\mathfrak N(\mathbf 1^k,\ell_k)|!).$$
 \end{lemma}
\begin{proof} Let $g=\gamma_q\dots \gamma_1\in U_k(\ell_k,0)$. From the definition of $U_k(\ell_k,0)$, it follows that, for any $w$ with $|w|=j-1\ge k-1$, 
each of the  bubble segment $$I_k(w)=B((w,a_j/2),a_j/2-\ell_k-1)$$ centered at  the middle points $(w, a_j/2)$ and of radius $a_j/2-\ell_k-1$
are left point-wise invariant by $g$. In fact, each of these segments is moved as a block throughout the sequence of steps 
$\gamma_j\dots,\gamma_1$, $1\le j\le q$  without escaping the full bubble segment containing $(w,a_j)$. Moreover, the translations of these segments 
are all following the moves of the point $\mathfrak m_k=(\mathbf 1^{k-1},a_k/2)$. This implies that these segments acts as buffers restricting the action of $g$
on points belonging to the various "connected components" of the complement of the union of these segments. Namely, the action of $g$ has to internal on each of these components. The statement of the lemma captures this fact and enumerate these components.    
\end{proof}

We now turn to the treatment of the pocket extension of $X_{\mathbf a}$.  Let us call the resulting group $\widetilde{\Gamma}_{\mathbf a}$. It is defined by the labelled graph $X^*_{\mathbf a}$ depicted schematically on Figure  \ref{B*} and generated by three elements $\tau,\alpha,\beta$.  The transposition $\tau$ transposes 
the new  vertex $*$ and  the root $\mathfrak o$ of $X_{\mathbf a}$. It act trivially at each of the other vertices (each carry a self-loop labelled $\tau$). 
At the new vertex $*$, the labellings $a$ and $b$ are carried by self-loops, i.e., $\alpha$ and $\beta$  act trivially at $*$. 
\begin{figure}[h]
\begin{center}\caption{Sketch 
of the Schreier graph $X^*_\mathbf a$ used to define $\widetilde{\Gamma}_{\mathbf a}$.}\label{B*}

\begin{picture}(150,60)(0,0) 

\put(30,0){\put(0,37){\makebox(0,0){$o$}} \put(-10,37){\makebox(0,0){$*$}}\put(0,30){\circle*{2}}
\put(-10,30){\circle*{2}}
\put(10,30){\oval(20,4.2)}
\put(22,30){\circle{4}}

\put(37.5,33){\oval(30,4.2)}
\put(37.5,27){\oval(30,4.2)}
\put(54,35.5){\circle{4}}
\put(54,24.2){\circle{4}}
\put(-9,32){\vector(1,0){8}}
\put(-1,28){\vector(-1,0){8}}
\put(-6,23){\makebox(0,0){$\tau$}}

\multiput(56,30)(10,0){7}{\line(1,0){5}} 

\put(10,39){\makebox(0,0){$a$}}
\put(8,35){\vector(1,0){8}}

\put(48,43){\makebox(0,0){$b$}}
\put(51,38){\vector(1,1){2}}}

\end{picture}\end{center}\end{figure}
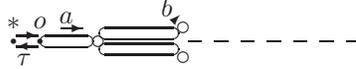

From this description, it should be rather obvious that exactly the same argument use for $\Gamma_{\mathbf a}$ applies to $\widetilde{\Gamma}_{\mathbf a}$
modulo some very small adaptation. For clarity, we  give explicitly the definition definition of the the test functions for $\widetilde{\Gamma}_{\mathbf a}$.
When working n $\widetilde{\Gamma}_{\mathbf a}$, the symmetric probability measure $\mathbf u$ is the uniform measure on $\{\tau,\alpha^{\pm 1},\beta^{\pm 1}\}$.

\begin{definition}
For each $k\ge 8$ and $\ell \in (0, (a_k/2)-1)$, consider the set $\widetilde{U}_k(\ell)$ of all elements $g\in \widetilde{\Gamma}_{\mathbf a}$ such that
there exists a sequence $\gamma_1,\dots \gamma_q\in \{\tau, \alpha^{\pm 1},\beta^{\pm 1}\}$ such that $g=\gamma_q\dots \gamma_1 $
and, for all $1\le j\le q$,  $\gamma_j\dots\gamma_1\mathfrak m_k \in B_k( \ell)$. 
For any $s\in \{0,1,\dots, \ell_k\}$, with $\ell_k = (a_k/4)-1 $,  define $\tilde{\psi}_{k}$ on $\Gamma_{\mathbf a}$ by
$$\tilde{\psi}_k(g)=\left\{ \begin{array}{cl} 0  &\mbox{ if } g\not\in \widetilde{U}_k(\ell_k)\\
(1- s/\ell_k)_+ & \mbox{ if } g\in \widetilde{U}_k(\ell_k) \mbox{ and }  d(\mathfrak m_k,g\mathfrak m_k)=s.\end{array}\right.$$ 
\end{definition}

\begin{lemma} \label{lem-bubble*}
The statements of {\em Lemma \ref{lem-bubbletest}} and {\em Lemma \ref{lem-testsupport}} apply to the functions $\widetilde{\psi}_k$
on $\widetilde{\Gamma}_{\mathbf a}$ after replacing $X_{\mathbf a}$ by $X^*_{\mathbf a}$ and  $U_k(\ell)$  by $\widetilde{U}_k(\ell)$. Note also that  
in the present case, the ball $B(\mathfrak o, r)$ is a ball in $X^*_{\mathbf a}$ and thus contains the extra vertex $*$. 
\end{lemma}

\begin{remark} It follows from Lemma \ref{lem-bubble*} that we have exactly the same upper-bound on the spectral profiles of the group 
$\Gamma_{\mathbf a}$ and $\widetilde{\Gamma}_{\mathbf a}$. Comparing with the results in \cite{SCZ-Rev},  this is also essentially the same upper-bound than for the permutation wreath-product  $\mathbb Z\wr_{X_{\mathbf a}} \Gamma_{\mathbf a}$. 
\end{remark}

Let $V_{\mathbf a}(t)=|B(\mathfrak o,t)|$ the volume of the ball or radius $t$ at the root on $X^{\mathbf a}$.
This volume is given by
$$V_{\mathbf a} (t)= \sum_1^{k-1} 2 a_j 2^{j-1} + 2(t-s_{k-1})2^{k-1}  \mbox{ for }   s_{k-1}\le t \le s_k.$$ 
This is the quantity that play a role when applying Corollary \ref{cor-*S} to obtain a lower bound on the isoperimetric and spectral profiles of $\widetilde{\Gamma}_{\mathbf a}$.  Namely,
$$\Lambda_{1,\widetilde{\Gamma}_{\mathbf a}} (v)\ge \frac{c}{r} \mbox{ whenever } v\le   \sqrt{V_{\mathbf a}(r)!}$$
which gives 
$$\Lambda _{1,\widetilde{\Gamma}_{\mathbf a}}(v)\gtrsim \frac{1}{V_{\mathbf a} ^{-1}( \frac{\log (1+v)}{\log(1+\log (1+v))})} \mbox{ where } V_{\mathbf a}^{-1}(v)=\inf\{ s: V_{\mathbf a}(s)\ge v\}.$$
The upper bound on $\Lambda_{2,\widetilde{\Gamma}_{\mathbf a}}(v) $ obtained above is based on the functions
$$W_{\mathbf a}(t) = \sum_1^{k-1} 2 a_j 2^{j-1} + (a_k/2) 2^{k-1}  \mbox{ for }   a_{k-1}< 2t \le a_{k}$$  and 
$$A_{\mathbf a}(t)= a_k/2   \mbox{ for }   a_{k-1}< 2t \le a_{k}.$$
It  reads
$$\Lambda_{2,\widetilde{\Gamma}_{\mathbf a}}(v)\le \frac{C}{r^2} \mbox{ whenever } v\ge   (W_{\mathbf a}(r)!)(A_{\mathbf a}(r)!) .$$
Because the factor  $A_{\mathbf a}(r)!$ is much smaller than the other factor and $r$ is on the scale of $\log v$, this gives
$$\Lambda_{2,\widetilde{\Gamma}_{\mathbf a}}(v)\le \frac{C}{[W_{\mathbf a}^{-1}( \frac{\log (1+v)}{\log(1+\log (1+v))})]^2}  \mbox{ where } W_{\mathbf a}^{-1}(v)=\sup\{ s: W_{\mathbf a}(s)\le v\} .$$ 

Taking into account  the left-hand side inequality in (\ref{Cheeger}),
the lower bound on the isoperimetric profile and the upper-bound on the spectral profile match-up rather well as long as  $a_k\simeq s_k$ 
(i.e., the sum $s_k$ is approximately equal to it last term $a_k$).  In the following theorem, we focus on the case when $a_k \simeq 2^{\kappa k}$ for some $\kappa\in (0,\infty)$.   In this case, we have  $s_k\simeq a_k$ and 
$$V_{\mathbf a} (r)\simeq W_{\mathbf a}(r) \simeq  r^{\frac{\kappa+1}{\kappa}}.$$

\begin{theorem} \label{bubblebound}  
Let $\widetilde{\Gamma}_{\mathbf{a}}$ be the group associated with $X^*_\mathbf a$, the Schreier graph pocket extension of  $X_{\mathbf a}$. Under the assumption that 
 $a_k\simeq 2^{\kappa k}$ for some $\kappa\in (0,1)$, the isoperimetric and spectral profiles satisfy 
$$\Lambda_{1,\widetilde{\Gamma}_{\mathbf a}}(v)^2\simeq  \Lambda_{2,\widetilde{\Gamma}_{\mathbf a}}(v)\simeq  \left(\frac{\log (1 +\log(1+v))}{\log (1+v)} \right)^{2\kappa/(\kappa+1))}.$$
The return probability function $\Phi_{\widetilde{\Gamma}_{\mathbf a}}$ satisfies
\[
\Phi_{\widetilde{\Gamma}_{\mathbf a}}(n)\simeq \exp\left(-n^{\frac{\kappa+1}{3\kappa+1}}(\log n)^{\frac{2\kappa}{3\kappa+1}}\right).
\]
\end{theorem}

It is perhaps surprising that the behavior of $\Lambda_{p,\Gamma_\mathbf a}$, $p=1,2$, and of $\Phi_{\Gamma_\mathbf a}$, for the bubble group $\Gamma_\mathbf a$ itself are not yet entirely understood. Because $\Gamma_{\mathbf a}$ is a subgroup of $\widetilde{\Gamma}_{\mathbf a}$ (and because the same arguments apply directly in both cases),  $\Lambda_{p,\Gamma_\mathbf a} \lesssim \Lambda_{p,\widetilde{\Gamma}_{\mathbf a}}$ and $\Phi_{\widetilde{\Gamma}_\mathbf a}\lesssim\Phi_{\Gamma_\mathbf a} $.   

Regarding the sequence $\mathbf a=(a_i)$ that defines $X_\mathbf a$, $\Gamma_\mathbf a$ and $\widetilde{\Gamma}_\mathbf a$, it is possible to obtain relatively good results for sequences $\mathbf a$ growing faster a than $ 2^{\kappa k}$. See \cite{SCZ-Rev} for related computations. Understanding the behavior of random walk on $\Gamma_\mathbf a$ and $\widetilde{\Gamma}_\mathbf a$ when the growth of $\mathbf a$ is slower than exponential appears to be a harder challenge.  The method explained here provides upper and lower bounds for $\Lambda_{p,\widetilde{\Gamma}_\mathbf a}$ and $ \Phi_{\widetilde{\Gamma}_\mathbf a}$ but theses bounds do not match. Again, see \cite{SCZ-Rev} for related computations.

\bibliographystyle{plain}

\end{document}